\documentclass{article}

\usepackage{float}
\floatstyle{ruled}
\newfloat{model}{thp}{lop}
\floatname{model}{Model}

\usepackage{multicol}

\usepackage{xcolor}

\newcommand{\refacpf}{(\hyperref[model:ac_pf]{AC-PF})}
\newcommand{\refsocpf}{(\hyperref[model:soc_pf]{SOC-PF})}
\newcommand{\refcdfpf}{(\hyperref[model:cdf_pf]{CDF-PF})}

\newcommand{\refacepf}{(\hyperref[model:ac_pf_e]{AC-E-PF})}
\newcommand{\refsocepf}{(\hyperref[model:soc_pf_e]{SOC-E-PF})}
\newcommand{\refcdfepf}{(\hyperref[model:cdf_pf_e]{CDF-E-PF})}

\usepackage{xspace}

\usepackage{url}
\usepackage{amssymb,amsmath,amsthm,bm,mathtools,mathabx}
\usepackage{graphicx}
\usepackage{authblk}

\usepackage{hyperref} 
\hypersetup{
colorlinks=true, 
breaklinks=true,
urlcolor= blue, 
linkcolor= blue,
citecolor=red, 
}

\newtheorem{theorem}{Theorem}[section]

\begin{document}

\title{DistFlow Extensions for AC Transmission Systems} 
\author[1]{Carleton Coffrin} 
\author[1]{Hassan L. Hijazi}
\author[2]{Pascal Van Hentenryck}

\affil[1]{Advanced Network Science Initiative, Los Alamos National Laboratory}
\affil[2]{Department of Industrial and Operations Engineering, University of Michigan}
\maketitle

\abstract{
Convex relaxations of the power flow equations and, in
particular, the Semi-Definite Programming (SDP), Second-Order
Cone (SOC), and Convex DistFlow (CDF) relaxations, have attracted 
significant interest in recent years. 
Thus far, studies of the CDF model and its connection to the other relaxations have been limited to power distribution systems, 
which omit several parameters necessary for modeling transmission systems.
To increase the applicability of the CDF relaxation, this paper develops an extended CDF model that is suitable for transmission systems by incorporating bus shunts, line charging, and transformers.
Additionally, a theoretical result shows that the established equivalence of the SOC and CDF models for distribution systems also holds in this transmission system extension.
}

\section*{Nomenclature}
\begin{multicols}{2} 

\begin{description}
  \item [{$N$}]  - The set of nodes in the network 
  \item [{$E$}]  - The set of {\em from} edges in the network 
  \item [{$E^R$}]  - The set of {\em to} edges in the network 
  %
  \item [{$\bm i$}] - imaginary number constant
  \item [{$I$}] - AC current
  \item [{$S = p+ \bm iq$}] - AC power
  \item [{$V = v \angle \theta$}]  - AC voltage
  \item [{$Z = r+ \bm ix$}] - Line impedance
  \item [{$Y = g + \bm ib$}]  - Line admittance
  \item [{$T = t \angle \theta^t$}]  - Transformer properties
  \item [{$Y^s = g^s + \bm ib^s$}]  - Bus shunt admittance
  \item [{$W $}]  - Product of two AC voltages
  \item [{$L$}]  - Current magnitude squared, $|I|^2$
  %
  \item [{$Y^c = g^c + \bm ib^c$}] - Line charging
  \item [{$s^u$}] - Line apparent power thermal limit
  \item [{$\theta^\Delta$}] - Voltage angle difference limit
  \item [{$S^d = p^d+ \bm iq^d$}] - AC power demand
  \item [{$S^g = p^g+ \bm iq^g$}] - AC power generation
  \item [{$c_0,c_1,c_2$}] - Generation cost coefficients 
 %
   \item [{$\Re(\cdot)$}] - Real component of a complex number
   \item [{$\Im(\cdot)$}] - Imaginary component of a complex number
   \item [{$(\cdot)^*$}] - Conjugate of a complex number
   \item [{$|\cdot|$}] - Magnitude of a complex number, $l^2$-norm
   \item [{$\angle$}] - Angle of a complex number
  %
  %
  \item [{$x^u$}] - Upper bound of $x$
  \item [{$x^l$}] - Lower bound of $x$
  \item [{$\widecheck{x}$}] - Convex envelope of $x$
  \item [{$\bm x$}] - A constant value
\end{description}

\end{multicols}

\clearpage
\section{Introduction}
\label{sec:intro}

Convex relaxations of the power flow equations have attracted
significant interest in recent years. They include the Semi-Definite
Programming (SDP) \cite{Bai2008383}, Second-Order Cone (SOC)
\cite{Jabr06}, Convex-DistFlow (CDF) \cite{6102366}, and the recent
Quadratic Convex (QC) \cite{qc_opf_tps,Hijazi2017} and Moment-Based \cite{7038397,6980142} relaxations.
Much of the excitement underlying this line of research comes from the fact that
the SDP relaxation has shown to be tight on a variety of case studies
\cite{5971792}, opening a new avenue for accurate, reliable, and
efficient solutions to a variety of power system applications. Indeed,
industrial-strength optimization tools (e.g., Gurobi \cite{gurobi},
cplex \cite{cplex}, Mosek \cite{mosek}) are now available to solve
various classes of convex optimization problems. 
In the context of power distribution systems, the relationships between the SDP, SOC, and CDF relaxations is well-understood \cite{6756976,6815671}.
In particular, the SOC and CDF relaxations are known to be equivalent \cite{6483453,6897933} and the QC and SDP relaxations are at least as strong
as both of these \cite{6345272, qc_opf_tps}. 

Thus far, studies of the CDF model and its connection to the other relaxations have been limited to power distribution systems, 
which omit several parameters necessary for modeling transmission systems (i.e. bus shunts, line charging, and transformers).
This paper develops an extended CDF model that incorporates all of the parameters necessary to use the CDF model on transmission system test cases \cite{nesta,pglib_opf}.
The main contributions of this work can be summarized as:
\begin{enumerate}
\item Developing a novel methodology for the derivation of the classic CDF model.
\item Utilizing that methodology to produce an extended CDF model that is applicable to transmission systems.
\item Proving that the equivalence of the SOC and CDF models carries over into the transmission system case.
\end{enumerate}

The rest of the paper is organized as follows.  
Section \ref{sec:ac_pf} reviews the formulation of the simplest AC power flow feasibility problem (AC-PF).
Section \ref{sec:ac_prop} develops a number of general properties of AC-PF that are utilized throughout the paper.
Section \ref{sec:ac_relax} reviews the two well known relaxations SOC and CDF for the AC-PF problem.
Section \ref{sec:ac_ex} is the heart of the paper and extends all of the previous results to an extended AC-PF formulation incorporating bus shunts, line charging, and transformers.
Lastly, the paper concludes with Section \ref{sec:conc}.

\section{AC Power Flow}
\label{sec:ac_pf}

This section reviews the specification of the AC Power Flow (AC-PF) feasibility problem
and introduces the notations used in the paper. In the equations,
bold face is used to reflect constants while script indicates variables.  
Capital letters refer to complex numbers while lower case are real numbers.

A power network is composed of a variety of components such as buses, lines, generators, and loads.  
The network can be interpreted as a graph $(N,E)$ where the set of buses $N$ represent the nodes and the set of lines $E$ represent the edges. 
Note that $E$ is a set of directed arcs and $E^R$ will be used to indicate those arcs in the reverse direction.
To break numerical symmetries in the model and to allow easy comparison of solutions, a reference node $r \in N$ is also specified.

The AC power flow equations are based on complex quantities for 
current $I$, voltage $V$, admittance
$Y$, and power $S$, which are linked by the physical properties of
Kirchhoff's Current Law (KCL), i.e.,
\begin{align}
& I^g_i - {\bm I^d_i} = \sum_{\substack{(i,j)\in E \cup E^R}} I_{ij}  \label{eq:kcl}
\end{align}
Ohm's Law, i.e.,
\begin{align}
& I_{ij} = \bm Y_{ij} (V_i - V_j)  \label{eq:ohm}
\end{align}
and the definition of AC power, i.e.,
\begin{align}
& S_{ij} = V_{i}I_{ij}^* \label{eq:power}
\end{align}
Combining these three properties yields the AC Power Flow equations, i.e.,
\begin{subequations}
\begin{align}
& S^g_i - {\bm S^d_i} = \sum_{\substack{(i,j)\in E \cup E^R}} S_{ij} \;\; \forall i\in N \label{eq:bus_s} \\ 
& S_{ij} = \bm Y^*_{ij} |V_i|^2 - \bm Y^*_{ij} V_i V^*_j \;\; (i,j)\in E \cup E^R \label{eq:line_s}
\end{align}
\end{subequations}

\noindent
These non-convex nonlinear equations define how power flows in the
network and are a core building block of many power system
applications. However, practical applications typically include various
operational side constraints on the flow of power. We now review some 
of the most significant ones.

\paragraph{Generator Capabilities}

AC generators have limitations on the amount of active and reactive
power they can produce $S^g$, which is characterized by a generation
capability curve \cite{9780070359581}.  Such curves typically define
nonlinear convex regions which are often approximated by boxes in
AC transmission system test cases, i.e.,
\begin{align}
& \bm {S^{gl}}_i \leq S^g_i \leq \bm {S^{gu}}_i \;\; \forall i \in N 
\end{align}

\paragraph{Line Thermal Limit}

AC power lines have thermal limits \cite{9780070359581} to prevent
lines from sagging and automatic protection devices from activating.
These limits are typically given in Volt Amp units and constrain 
the apparent power flows on the lines, i.e.,
\begin{align}
& |S_{ij}| \leq \bm {s^u}_{ij} \;\; \forall (i,j) \in E \cup E^R 
\end{align}

\paragraph{Bus Voltage Limits}

Voltages in AC power systems should not vary too far (typically $\pm
10\%$) from some nominal base value \cite{9780070359581}.  This is
accomplished by putting bounds on the voltage magnitudes, i.e.,
\begin{align}
& \bm {v^l}_i \leq |V_i| \leq \bm {v^u}_i \;\; \forall i \in N
\end{align}
%

\paragraph{Voltage Angle Differences}

Small voltage angle differences are also a design imperative in AC power
systems \cite{9780070359581} and it has been suggested that phase
angle differences are typically less than $10$ degrees in practice
\cite{Purchala:2005gt}. These constraints have not typically been
incorporated in AC transmission test cases \cite{matpower}. However,
recent work \cite{LPAC_ijoc,Hijazi2017,7763860} have observed that
incorporating Voltage Angle Difference (VAD) constraints, i.e.,
\begin{align}
&  -\bm {\theta^\Delta}_{ij} \leq \angle \! \left( V_i V^*_j \right) \leq \bm {\theta^\Delta}_{ij} \;\; \forall (i,j) \in E \label{eq:pad_1}
\end{align}
is useful in the convexification of the AC power flow equations. For
simplicity, this paper assumes that the voltage angle difference bounds
are symmetrical and within the range $(- \bm \pi/2, \bm \pi/2 )$, i.e.,
\begin{align}
& 0 \leq \bm {\theta^{\Delta}}_{ij} \leq \frac{\bm \pi}{2} \;\; (i,j) \in E \nonumber
\end{align}
but the results presented here can be extended to more general
cases. Observe also that the VAD constraints \eqref{eq:pad_1} can be
implemented as a linear relation of the real and imaginary components
of $V_iV^*_j$ \cite{6822653}, i.e.,
\begin{align}
& \tan(-\bm {\theta^\Delta}_{ij}) \Re\left(V_iV^*_j\right) \leq \Im\left(V_iV^*_j\right) \leq \tan(\bm {\theta^\Delta}_{ij}) \Re\left(V_iV^*_j\right) \;\; \forall (i,j) \in E \label{eq:w_pad}
\end{align}
and that equation \eqref{eq:line_s} can be used to express these in terms of the $S$ variables as follows,
\begin{align}
& V_iV^*_j = |V_i|^2 - \bm Z^*_{ij} S_{ij} \label{eq:vv_factor}
\end{align}
These equations combined with \eqref{eq:w_pad} implement the VAD constraints in terms of the $V$ and $S$ variables as follows,  
\begin{align}
& \tan(-\bm {\theta^\Delta}_{ij}) \Re\left( |V_i|^2 - \bm Z^*_{ij} S_{ij} \right) \leq \Im\left( |V_i|^2 - \bm Z^*_{ij} S_{ij} \right) \leq \tan(\bm {\theta^\Delta}_{ij}) \Re\left( |V_i|^2 - \bm Z^*_{ij} S_{ij} \right) \;\; \forall (i,j) \in E \label{eq:s_pad}
\end{align}
The usefulness of these alternate formulations will be apparent later in the paper.

\paragraph{Other Constraints}
Other line flow constraints have been proposed, such as, active power
limits and voltage difference limits \cite{5971792,6822653}.  However,
we do not consider them here since, to the best of our knowledge, test
cases incorporating these constraints are not readily available.

%

\begin{model}[t]
\caption{The AC Power Flow Feasibility Problem (AC-PF)}
\label{model:ac_pf}
\begin{subequations}
\begin{align}
\mbox{\bf variables:} \nonumber \\
& S^g_i \;\; \forall i\in N \nonumber \\
& V_i \;\; \forall i\in N \nonumber \\
& S_{ij} \;\; \forall(i,j)\in E \cup E^R \nonumber \\
%
%
\mbox{\bf subject to:} \nonumber \\
& \angle V_{\bm r} = 0 \label{eq:ac_0} \\
& \bm {v^l}_i \leq |V_i| \leq \bm {v^u}_i \;\; \forall i \in N \label{eq:ac_1} \\
& \bm {S^{gl}}_i \leq S^g_i \leq \bm {S^{gu}}_i \;\; \forall i \in N \label{eq:ac_2}  \\
& |S_{ij}| \leq \bm {s^u}_{ij} \;\; \forall (i,j) \in E \cup E^R \label{eq:ac_5}  \\
& S^g_i - {\bm S^d_i} = \sum_{\substack{(i,j)\in E \cup E^R}} S_{ij} \;\; \forall i\in N \label{eq:ac_3}  \\ 
& S_{ij} = \bm Y^*_{ij} |V_i|^2 - \bm Y^*_{ij} V_i V^*_j \;\; (i,j)\in E \cup E^R \label{eq:ac_4}  \\
& -\bm {\theta^\Delta}_{ij} \leq \angle (V_i V^*_j) \leq \bm {\theta^\Delta}_{ij} \;\; \forall (i,j) \in E  \label{eq:ac_6} 
\end{align}
\end{subequations}
\end{model}

\paragraph{The AC Power Flow Feasibility Problem}

Combining all of these constraints yields the AC Power Flow Feasibility presented in Model \ref{model:ac_pf}. 
Constraint \eqref{eq:ac_0} sets the reference angle, to eliminate numerical symmetries.  
Constraints \eqref{eq:ac_3} capture KCL and constraints \eqref{eq:ac_4} capture Ohm's Law.  
Constraints \eqref{eq:ac_1} and \eqref{eq:ac_6} capture the voltage magnitude and voltage angle difference operational constraints.
Finally constraints \eqref{eq:ac_2} and \eqref{eq:ac_5} enforce the generator output and line flow limits. 
Notice that this is a non-convex nonlinear satisfaction problem due to the product of voltage variables, $V_i V^*_j$, and is NP-Hard in general \cite{verma2009power,ACSTAR2015}.

\section{Generic Properties of AC Power Flows}
\label{sec:ac_prop}

Before deriving the various relaxations of the AC power flow equations, this section develops a collection of general properties of the AC power flows \refacpf, which are utilized throughout the rest of the paper. 

\noindent
\paragraph{Absolute Square of Ohm's Law}
From definition \eqref{eq:ohm}, the absolute square of Ohm's law is,
\begin{subequations}
\begin{align}
& I_{ij}I_{ij}^*  = (\bm Y_{ij} V_i - \bm Y_{ij} V_j) (\bm Y_{ij}^*  V_i^* - \bm Y_{ij}^*  V_j^*) \\
& |I_{ij}|^2 = |\bm Y_{ij}|^2 (|V_i|^2  - V_iV_j^* - V_i^*V_j + |V_j|^2)
\end{align}
\end{subequations}
\noindent
\paragraph{Absolute Square of AC Power}
From definition \eqref{eq:power}, the absolute square of the definition of ac power is,
\begin{subequations}
\begin{align}
& S_{ij}S_{ij}^* = (V_iI_{ij}^*)(V_i^*I_{ij}) \\
& |S_{ij}|^2 = |V_i|^2 |I_{ij}|^2 \label{eq:asp}
\end{align}
\end{subequations}
\noindent
\paragraph{Absolute Square of Voltage Products}
The absolute square of the voltage product is,
\begin{subequations}
\begin{align}
& (V_iV_j^*)(V_iV_j^*)^* = (V_iV_j^*)(V_i^*V_j) \\ 
& |V_iV_j^*|^2 = |V_i|^2 |V_j|^2 \label{eq:asvp}
\end{align}
\end{subequations}
\paragraph{Line Loss}
From definition \eqref{eq:line_s}, the power loss on line $(i,j)$ is,
\begin{subequations}
\begin{align}
& S_{ij}+S_{ji} = \bm Y^*_{ij} (|V_i|^2 - V_iV_j^* - V_i^*V_j + |V_j|^2 )  \label{eq:v_loss}
\end{align}
\end{subequations}
\paragraph{Voltage Magnitude Difference}
From definition \eqref{eq:line_s}, the so-called voltage drop property is derived by subtracting the power on each end of a line $(i,j)$, solving for the $|V_i|^2,|V_j|^2$ variables, and eliminating the $S_{ji}$ term.
\begin{subequations}
\begin{align}
& S_{ij}-S_{ji} = \bm Y^*_{ij}(|V_i|^2 - V_iV_j^* + V_i^*V_j - |V_j|^2) \\
& |V_i|^2 - |V_j|^2 = \bm Z^*_{ij} S_{ij} - \bm Z^*_{ij} S_{ji} + V_iV_j^* - V_i^*V_j \\
& |V_i|^2 - |V_j|^2 = \bm Z^*_{ij} S_{ij} - \bm Z^*_{ij} (\bm Y^*_{ij} (|V_i|^2 - V_iV_j^* - V_i^*V_j + |V_j|^2) - S_{ij})  + V_iV_j^* - V_i^*V_j \\
& |V_i|^2 - |V_j|^2 = \bm Z^*_{ij} S_{ij} - (|V_i|^2 - V_iV_j^* - V_i^*V_j + |V_j|^2) + \bm Z^*_{ij} S_{ij}  + V_iV_j^* - V_i^*V_j \\
& |V_i|^2 - |V_j|^2 = \bm Z^*_{ij} S_{ij} - (|V_i|^2 - V_iV_j^* - V_i^*V_j + |V_j|^2 ) + (|V_i|^2 - V_iV_j^*)  + V_iV_j^* - V_i^*V_j \\
& |V_i|^2 - |V_j|^2 = (\bm Z^*_{ij} S_{ij} + \bm Z_{ij} S_{ij}^*) - (|V_i|^2  - V_iV_j^* - V_i^*V_j + |V_j|^2) \label{eq:v_drop}
\end{align}
\end{subequations}
\paragraph{Equivalence of Line Flow Formulations}
A key observation from these general power flow properties is that the line loss and voltage magnitude difference properties provide an alternate formulation of the power flow constraints \eqref{eq:line_s}, namely,
\begin{align}
& \eqref{eq:line_s} \Leftrightarrow \eqref{eq:v_loss}, \eqref{eq:v_drop} \nonumber
\end{align}
which follows from the fact that these two sets of constraints are simply linear combinations of one another.

\section{Derivation of the Relaxations}
\label{sec:ac_relax}

This section derives two well known convex relaxations of Model \ref{model:ac_pf}, the SOC relaxation \cite{Jabr06} and the Convex DistFlow relaxation \cite{6102366} and reviews their equivalence.  Although none of the results in this section are new, the derivation of the models presents a seemingly novel and systematic methodology that is leveraged in subsequent sections for developing the extended Convex DistFlow relaxation for transmission systems.

\subsection{The SOC Relaxation}
\label{sec:soc_relax}

The SOC relaxation was first proposed in \cite{Jabr06} and utilizes two key insights.  First, by lifting the product of voltage variables $V_i V_j^*$ into a higher dimensional space (i.e. the $W$-space),
\begin{subequations}
\begin{align}
& W_{i} = |V_{i}|^2 \;\; i \in N \label{eq:w_link_1} \\
& W_{ij} = V_i V_j^* \;\; \forall(i,j) \in E \label{eq:w_link_2} 
\end{align}
\end{subequations}
a convex relaxation of \refacpf~is obtained.  Second, the absolute square of voltage products property \eqref{eq:asvp} can be used to strengthen this $W$-space relaxation, as follows,
\begin{subequations}
\begin{align}
& |W_{ij}|^2 = W_i W_j \;\; \forall(i,j) \in E \\
& |W_{ij}|^2 \leq W_i W_j \;\; \forall(i,j) \in E \label{eq:asvp_w}
\end{align}
\end{subequations}
Notice that constraint \eqref{eq:asvp_w} is a convex second-order cone constraint, which is widely supported by industrial strength convex optimization tools (e.g., Gurobi \cite{gurobi}, CPlex \cite{cplex}, Mosek \cite{mosek}).

The complete SOC relaxation of \refacpf~is presented in Model \ref{model:soc_pf} \refsocpf.
The constraints for KCL \eqref{eq:ac_3}, generator output limits \eqref{eq:ac_2}, and line flow limits \eqref{eq:ac_5} are identical to the \refacpf~model.
Constraints \eqref{eq:soc_3}--\eqref{eq:soc_4} capture line power flow in the $W$-space.  
Constraints \eqref{eq:soc_1} and \eqref{eq:soc_2} capture the voltage and voltage angle difference operational constraints.
Finally constraints \eqref{eq:soc_5} strengthen the relaxation with voltage product second-order cone constraint.

\begin{model}[t]
\caption{The SOC Relaxation of AC Power Flow (SOC-PF)}
\label{model:soc_pf}
\begin{subequations}
\begin{align}
\mbox{\bf variables:} \nonumber \\
& S^g_i \;\; \forall i\in N \nonumber \\
& W_i \;\; \forall i\in N \nonumber \\
& W_{ij} \;\; \forall (i,j)\in E \nonumber \\
& S_{ij} \;\; \forall (i,j)\in E \cup E^R \nonumber \\
\mbox{\bf subject to:} \nonumber \\
& \eqref{eq:ac_2}, \eqref{eq:ac_5}, \eqref{eq:ac_3} \nonumber \\ 
& (\bm {v^l}_i)^2 \leq W_i \leq (\bm {v^u}_i)^2 \;\; \forall i \in N \label{eq:soc_1} \\
& \tan(-\bm {\theta^\Delta}_{ij}) \Re\left(W_{ij}\right) \leq \Im\left(W_{ij}\right) \leq \tan(\bm {\theta^\Delta}_{ij}) \Re\left(W_{ij}\right) \;\; \forall (i,j) \in E \label{eq:soc_2} \\
& S_{ij} = \bm Y^*_{ij} W_i - \bm Y^*_{ij} W_{ij} \;\; (i,j)\in E \label{eq:soc_3}  \\
& S_{ji} = \bm Y^*_{ij} W_j - \bm Y^*_{ij} W_{ij}^* \;\; (i,j)\in E \label{eq:soc_4}  \\
& |W_{ij}|^2 \leq W_i W_j  \;\; \forall (i,j) \in E \label{eq:soc_5}
\end{align}
\end{subequations}
\end{model}

\subsection{The Convex DistFlow Relaxation}
\label{sec:df_relax}

The DistFlow (DF) model is a non-convex power flow model originally developed in \cite{19266} and later the Convex DistFlow (CDF) model was proposed in \cite{6102366}.  Both DF and CDF were originally designed for radial topologies, however, \cite{6507355,6507352} show that they can be interpreted as a phase-angle relaxation of the meshed AC power flow and hence are applicable to the study of meshed power networks.  This section presents a novel derivation of the DF model and its relaxation from first principles.  The value of this derivation is to establish a clear connection between \refacpf~and the CDF model.

The DistFlow model \cite{19266} can be derived utilizing three key insights.  First, replace the line flow equations \eqref{eq:line_s} with their alternate formulation based on line loss \eqref{eq:v_loss} and voltage magnitude difference \eqref{eq:v_drop}.  Second, lift the model into the space of bus voltages and line currents (i.e. the $L$-space),
\begin{subequations}
\begin{align}
& W_{i} = |V_{i}|^2 \;\; i \in N \label{eq:w_link_1} \\
& L_{ij} = |\bm Y_{ij}|^2 (|V_i|^2  - V_iV_j^* - V_i^*V_j + |V_j|^2 ) \;\; \forall(i,j) \in E \label{eq:w_link_2} 
\end{align}
\end{subequations}
Third, use the absolute square of power property \eqref{eq:asp} to strengthen the $L$-space relaxation, as follows,
\begin{align}
& |S_{ij}|^2 = W_i L_{ij} \;\; \forall(i,j) \in E \label{eq:asp_l_ncvx}
\end{align}
This establishes the non-convex DF model.  The convex relaxation, CDF, is obtained by relaxing \eqref{eq:asp_l_ncvx} to an inequality \cite{6102366}, 
\begin{align}
& |S_{ij}|^2 \leq W_i L_{ij} \;\; \forall(i,j) \in E \label{eq:asp_l} 
\end{align}
Notice that constraint \eqref{eq:asp_l} is a second-order cone constraint.

\begin{model}[t]
\caption{The CDF Relaxation of AC Power Flow (CDF-PF)}
\label{model:cdf_pf}
\begin{subequations}
\begin{align}
\mbox{\bf variables:} \nonumber \\
& S^g_i \;\; \forall i\in N \nonumber \\
& W_i \;\; \forall i\in N \nonumber \\
& L_{ij} \;\; \forall (i,j)\in E \nonumber \\
& S_{ij} \;\; \forall (i,j)\in E \cup E^R \nonumber \\
\mbox{\bf subject to:} \nonumber \\
& \eqref{eq:ac_2}, \eqref{eq:ac_5}, \eqref{eq:ac_3}, \eqref{eq:soc_1} \nonumber \\ 
& \Im\left( W_i - \bm Z^*_{ij} S_{ij} \right) \leq \tan(\bm {\theta^\Delta}_{ij}) \Re\left( W_i - \bm Z^*_{ij} S_{ij} \right) \;\; \forall (i,j) \in E \label{eq:cdf_1_1}  \\
& \Im\left( W_i - \bm Z^*_{ij} S_{ij} \right) \geq \tan(- \bm {\theta^\Delta}_{ij}) \Re\left( W_i - \bm Z^*_{ij} S_{ij} \right) \;\; \forall (i,j) \in E \label{eq:cdf_1_2}  \\
& S_{ij} + S_{ji} = \bm Z_{ij} L_{ij}  \;\; (i,j)\in E \label{eq:cdf_2}  \\
& W_i - W_j = (\bm Z^*_{ij} S_{ij} + \bm Z_{ij} S_{ij}^*) -  |\bm Z_{ij}|^2 L_{ij}  \;\; (i,j)\in E \label{eq:cdf_3}  \\
& |S_{ij}|^2 \leq W_i L_{ij}  \;\; \forall (i,j) \in E \label{eq:cdf_4} 
\end{align}
\end{subequations}
\end{model}

The complete CDF relaxation of \refacpf~is presented in Model \ref{model:cdf_pf} \refcdfpf.
The constraints for KCL \eqref{eq:ac_3}, generator output limits \eqref{eq:ac_2}, and line flow limits \eqref{eq:ac_5} are identical to the AC-PF model and constraints \eqref{eq:soc_1} for the voltage bounds are identical to the \refsocpf~model.
Constraints \eqref{eq:cdf_2}--\eqref{eq:cdf_3} capture line power flow in the $L$-space.  
Constraints \eqref{eq:cdf_1_1}--\eqref{eq:cdf_1_2} capture the voltage angle difference constraints, along the lines of \eqref{eq:s_pad}.
Finally constraints \eqref{eq:cdf_4} strengthen the relaxation with the power-based second-order cone constraint.

\subsection{Equivalence of the Relaxations}

Interestingly, \cite{6483453} showed that the \refsocpf~and \refcdfpf~relaxations are equivalent.
The key insight is that a bijection between the solution sets of \refsocpf~and \refcdfpf~can be developed as follows.
Given a solution to the \refcdfpf~($W_i, L_{ij}, S_{ij}$), assign the \refsocpf~variables,
\begin{subequations}
\begin{align}
& W_{ij} = W_i - \bm Z^*_{ij} S_{ij} \;\; (i,j) \in E
\end{align}
\end{subequations}
Given a solution to the \refsocpf~($W_i, W_{ij}, S_{ij}$), assign the \refcdfpf~variables,
\begin{subequations}
\begin{align}
& L_{ij} = |\bm{Y}_{ij}|^2 \left( W_i - W_{ij} - W_{ij}^* + W_j  \right)  \;\; (i,j) \in E
\end{align}
\end{subequations}
Note that $W$ and $S$ remain the same in both models. In each of these assignments, all of the model constraints are satisfied, demonstrating the bijection.  See \cite{6897933,6483453} for a detailed proof.

\section{Extensions of the Power Flow Models}
\label{sec:ac_ex}

In the interest of clarity, AC Power Flows, and their relaxations, are most often presented on the purest version of the AC power flow equations.  However, industrial transmission system datasets include additional parameters such as bus shunts ($\bm {Y^s}$), asymmetric line charging ($\bm {Y^c}_{ij}, \bm {Y^c}_{ji}$), and transformers ($\bm T$), which complicate the AC power flow equations significantly.
This section extends the results of the previous sections to include these additional parameters.

We begin with the definitions of the network's physical properties.
Kirchhoff's Current Law (KCL) is extended to incorporate bus shunts as follows,
\begin{align}
& I^g_i - {\bm I^d_i} - {\bm Y^s_i} V_i = \sum_{\substack{(i,j)\in E \cup E^R}} I_{ij} \;\; \forall i \in N \label{eq:kcl_e}
\end{align}
Ohm's Law in extended to include transformers and line charging,
\begin{subequations}
\begin{align}
& I_{ij} = \left( \bm Y_{ij} + \bm {Y^c}_{ij} \right) \frac{V_i}{\bm T_{ij}} - \bm Y_{ij} V_j  \;\; \forall (i,j)\in E \label{eq:ohm_e} \\
& I_{ji} = \left( \bm Y_{ij} + \bm {Y^c}_{ji} \right) V_j - \bm Y_{ij} \frac{V_i}{\bm T_{ij}}  \;\; \forall (i,j)\in E \label{eq:ohm_e}
\end{align}
\end{subequations}
and the definition of AC power now includes the voltage transformation,
\begin{align}
& S_{ij} = \frac{V_i}{\bm T_{ij}} I_{ij}^*  \;\; \forall (i,j)\in E \label{eq:power_e}
\end{align}
Combining these three properties yields the extended AC Power Flow equations,
\begin{subequations}
\begin{align}
& S^g_i - {\bm S^d_i} - (\bm {Y^s}_i)^* |V_i|^2 = \sum_{\substack{(i,j)\in E \cup E^R}} S_{ij} \;\; \forall i\in N \label{eq:bus_s_e} \\ 
& S_{ij} = \left( \bm Y_{ij} + \bm {Y^c}_{ij} \right)^* \frac{|V_i|^2}{|\bm{T}_{ij}|^2} - \bm Y^*_{ij} \frac{V_iV^*_j}{\bm{T}_{ij}} \;\; (i,j)\in E \label{eq:line_s_e_from}\\
& S_{ji} = \left( \bm Y_{ij} + \bm {Y^c}_{ji} \right)^* |V_j|^2 - \bm Y^*_{ij} \frac{V^*_iV_j}{\bm{T}^*_{ij}} \;\; (i,j)\in E \label{eq:line_s_e_to}
\end{align}
\end{subequations}
The complete Extended AC Power Flow Feasibility Problem \refacepf~is presented in Model \ref{model:ac_pf_e}. 
The operational constraints \eqref{eq:ac_0}--\eqref{eq:ac_5} and \eqref{eq:ac_6} remain the same as \refacpf.  
Constraints \eqref{eq:ac_e_1} capture the extended KCL and constraints \eqref{eq:ac_e_2}--\eqref{eq:ac_e_3} capture the extended Ohm's Law.  

\begin{model}[t]
\caption{The Extended AC Power Flow Feasibility Problem (AC-E-PF)}
\label{model:ac_pf_e}
\begin{subequations}
\begin{align}
\mbox{\bf variables:} \nonumber \\
& \mbox{Variables from \refacpf} \nonumber \\
\mbox{\bf subject to:} \nonumber \\
& \mbox{ \eqref{eq:ac_0},\eqref{eq:ac_1},\eqref{eq:ac_2},\eqref{eq:ac_5},\eqref{eq:ac_6} } \nonumber \\
& S^g_i - {\bm S^d_i} - (\bm {Y^s}_i)^* |V_i|^2 = \sum_{\substack{(i,j)\in E \cup E^R}} S_{ij} \;\; \forall i\in N \label{eq:ac_e_1} \\ 
& S_{ij} = \left( \bm Y_{ij} + \bm {Y^c}_{ij} \right)^* \frac{|V_i|^2}{|\bm{T}_{ij}|^2} - \bm Y^*_{ij} \frac{V_iV^*_j}{\bm{T}_{ij}} \;\; (i,j)\in E \label{eq:ac_e_2}\\
& S_{ji} = \left( \bm Y_{ij} + \bm {Y^c}_{ji} \right)^* |V_j|^2 - \bm Y^*_{ij} \frac{V^*_iV_j}{\bm{T}^*_{ij}} \;\; (i,j)\in E   \label{eq:ac_e_3}
\end{align}
\end{subequations}
\end{model}

It is useful to observe that even in this extended formulation, the product of voltages can be factored in terms of the $V$ and $S$ variables, as done previously in \eqref{eq:vv_factor},
\begin{align}
& V_iV^*_j = \bm Z^*_{ij}  \bm{T}_{ij} \left( \left( \bm Y_{ij} + \bm {Y^c}_{ij} \right)^* \frac{|V_i|^2}{|\bm T_{ij}|^2} - S_{ij} \right) \;\; (i,j) \in E \label{eq:vv_factor_e}
\end{align}
This factorization is useful for extending properties such as the VAD constraints, along the lines of \eqref{eq:s_pad}, and for proving the equivalence of various relaxations.

\subsection{Extensions of the General Properties}
\label{sec:ac_prop_e}

Next the general network properties from Section \ref{sec:ac_prop} are redefined in this extended AC power flow context.

\noindent
\paragraph{Absolute Square of Ohm's Law}
From definition \eqref{eq:ohm_e}, the absolute square of Ohm's law is,
\begin{subequations}
\begin{align}
& I_{ij}I_{ij}^*  = \left(  \left( \bm Y_{ij} + \bm {Y^c}_{ij}  \right) \frac{V_i}{\bm T_{ij}} - \bm Y_{ij} V_j  \right) \left(  \left( \bm Y_{ij} - \bm {Y^c}_{ij} \right)^* \frac{V_i^*}{\bm T^*_{ij}} - \bm Y_{ij}^* V_j^*  \right)  \\
& \mbox{resulting in,} \nonumber \\
& |I_{ij}|^2 = |\bm Y_{ij}|^2 \left( \frac{|V_i|^2}{|\bm T_{ij}|^2} - \frac{V_iV_j^*}{\bm T_{ij}} - \frac{V_i^*V_j}{\bm T^*_{ij}} + |V_j|^2  \right) 
   + (\bm {Y^c}_{ij} S_{ij} +  \bm {Y^{c*}}_{ij} S^*_{ij})
   - |\bm {Y^c}_{ij}|^2 \frac{|V_i|^2}{|\bm T_{ij}|^2} 
\end{align}
\end{subequations}
A complete derivation appears in Appendix \ref{sec:prop_e_proof}.
\noindent
\paragraph{Absolute Square of AC Power}
From definition \eqref{eq:power_e}, the absolute square of the definition of ac power is,
\begin{subequations}
\begin{align}
& S_{ij}S_{ij}^* = \left( \frac{V_i}{\bm T_{ij}} I_{ij}^* \right) \left( \frac{V_i^*}{\bm T^*_{ij}} I_{ij} \right)  \\
& |S_{ij}|^2 = \frac{|V_i|^2}{|\bm T_{ij}|^2} |I_{ij}|^2 \label{eq:asp_e}
\end{align}
\end{subequations}
\noindent
\paragraph{Absolute Square of Voltage Products}
The absolute square of voltage products is,
\begin{subequations}
\begin{align}
& \left( \frac{V_i}{\bm T_{ij}} V_j^* \right) \left( \frac{V_i}{\bm T_{ij}} V_j^* \right)^* = \left( \frac{V_i}{\bm T_{ij}} V_j^* \right) \left( \frac{V_i^*}{\bm T^*_{ij}} V_j \right) \\ 
& \frac{|V_i  V_j^*|^2}{|\bm T_{ij}^*|^2} = \frac{|V_i|^2}{|\bm T_{ij}^*|^2} |V_j|^2 \\
& |V_iV_j^*|^2 = |V_i|^2 |V_j|^2
\end{align}
\end{subequations}
The derivation follows similarly to the previous section.  Interestingly, the transformer constants cancel and this property is unaffected by the model extensions.
\paragraph{Line Losses}
From the definitions \eqref{eq:line_s_e_from}--\eqref{eq:line_s_e_to}, the power loss on line $(i,j)$ is,
\begin{subequations}
\begin{align}
& S_{ij}+S_{ji} = \bm Y^*_{ij} \left( \frac{|V_i|^2}{|\bm T_{ij}|^2} - \frac{V_iV_j^*}{\bm T_{ij}} - \frac{V_i^*V_j}{\bm T^*_{ij}} + |V_j|^2  \right) +  \bm {Y^{c*}}_{ij} \frac{|V_i|^2}{|\bm T_{ij}|^2} + \bm {Y^{c*}}_{ji} |V_j|^2  \label{eq:v_loss_e}
\end{align}
\end{subequations}
\paragraph{Voltage Magnitude Difference}
From the definitions \eqref{eq:line_s_e_from}--\eqref{eq:line_s_e_to}, the extended voltage magnitude difference property is derived by subtracting the power on each end of a line $(i,j)$, solving for the $|V_i|^2,|V_j|^2$ variables, and eliminating the $S_{ji}$ term.
\begin{subequations}
\begin{align}
& S_{ij} - S_{ji}  =  \bm {Y^{c*}}_{ij} \frac{|V_i|^2}{|\bm{T}_{ij}|^2} - \bm {Y^{c*}}_{ji} |V_j|^2  + \bm Y^*_{ij} \left(\frac{|V_i|^2}{|\bm{T}_{ij}|^2} - |V_j|^2 \right) - \bm Y^*_{ij} \left( \frac{V_iV_j^*}{\bm{T}_{ij}} - \frac{V_i^*V_j}{\bm{T}^*_{ij}} \right)  \\
& \mbox{resulting in,} \nonumber \\
& \left(1 + \bm Z_{ij}\bm {Y^c}_{ij} + \bm Z^*_{ij}\bm {Y^{c*}}_{ij} \right) \frac{|V_i|^2}{|\bm T_{ij}|^2}-|V_j|^2 = (\bm Z^*_{ij} S_{ij} + \bm Z_{ij} S_{ij}^*) - \left( \frac{|V_i|^2}{|\bm T_{ij}|^2} - \frac{V_iV_j^*}{\bm T_{ij}} - \frac{V_i^*V_j}{\bm T^*_{ij}} + |V_j|^2 \right)  \label{eq:v_drop_e} 
\end{align}
\end{subequations}
A complete derivation appears in Appendix \ref{sec:prop_e_proof}.
\paragraph{Equivalence of Line Flow Formulations}
Despite these various extensions, the line loss and voltage magnitude difference constraints still provide an alternate formulation of the power flow constraints \eqref{eq:line_s_e_from}--\eqref{eq:line_s_e_to} namely,
\begin{align}
& \mbox{\eqref{eq:line_s_e_from},\eqref{eq:line_s_e_to}} \Leftrightarrow \eqref{eq:v_loss_e},\eqref{eq:v_drop_e} \nonumber 
\end{align}
which follows from the fact that these two sets of constraints are simply linear combinations of one another.

\subsection{The Extended SOC Relaxation}
\label{sec:soc_relax_e}

The Extended SOC relaxation of \refacepf~is presented in Model \ref{model:soc_pf_e} \refsocepf.
The relaxation is nearly identical to the simple version as the voltage variables $V$ and the voltage product property are unaffected by the model extensions.  
Simply lifting \refacepf~into the $W$-space and strengthening with the original second-order cone constraint \eqref{eq:soc_5} completes the relaxation.

\begin{model}[t]
\caption{The SOC Relaxation of Extended AC Power Flow (SOC-E-PF)}
\label{model:soc_pf_e}
\begin{subequations}
\begin{align}
\mbox{\bf variables:} \nonumber \\
& \mbox{Variables from \refsocpf } \nonumber \\
\mbox{\bf subject to:} \nonumber \\
& \eqref{eq:ac_2}, \eqref{eq:ac_5}, \eqref{eq:soc_1}, \eqref{eq:soc_2}, \eqref{eq:soc_5}  \nonumber \\ 
& S^g_i - {\bm S^d_i} - (\bm {Y^s}_i)^* W_i = \sum_{\substack{(i,j)\in E \cup E^R}} S_{ij} \;\; \forall i\in N \label{eq:soc_e_1} \\ 
& S_{ij} = \left( \bm Y_{ij} + \bm {Y^c}_{ij} \right)^* \frac{W_i}{|\bm{T}_{ij}|^2}  - \bm Y^*_{ij}\frac{W_{ij}}{\bm{T}_{ij}}  \;\; (i,j)\in E \label{eq:soc_e_2}  \\
& S_{ji} =  \left( \bm Y_{ij} + \bm {Y^c}_{ji} \right)^* W_j - \bm Y^*_{ij}\frac{W_{ij}^*}{\bm{T}^*_{ij}} \;\; (i,j)\in E \label{eq:soc_e_3} 
\end{align}
\end{subequations}
\end{model}

\subsection{The Extended Convex DistFlow Relaxation}
\label{sec:df_relax_e}

Thus far, applications of this DF model have typically focused on distribution systems and, to the best of our knowledge, the DF model has not been extended to capture line charging and transformers. Using Section \ref{sec:df_relax} as a guide, this section derives an extended DF model for AC transmission systems featuring bus shunts, line charging, and transformers.

First, we replace the line flow equations \eqref{eq:line_s_e_from}--\eqref{eq:line_s_e_to} with their alternate formulation based on line loss \eqref{eq:v_loss_e} and voltage magnitude difference \eqref{eq:v_drop_e}.  
Second, we lift the model into the $L$-space, as follows,
\begin{subequations}
\begin{align}
& W_{i} = |V_{i}|^2 \;\; i \in N \label{eq:w_link_e_1} \\
& L_{ij} = |\bm Y_{ij}|^2 \left( \frac{|V_i|^2}{|\bm T_{ij}|^2} - \frac{V_iV_j^*}{\bm T_{ij}} - \frac{V_i^*V_j}{\bm T^*_{ij}} + |V_j|^2  \right) 
  + (\bm {Y^c}_{ij} S_{ij} +   \bm {Y^{c*}}_{ij} S^*_{ij})
  - |\bm {Y^c}_{ij}|^2 \frac{|V_i|^2}{|\bm T_{ij}|^2} \;\; \forall(i,j) \in E \label{eq:w_link_e_2} 
\end{align}
\end{subequations}
Third, use the absolute square of power property \eqref{eq:asp_e} to strengthen the $L$-space relaxation,
\begin{align}
& |S_{ij}|^2 = \frac{W_i}{|\bm T_{ij}|^2} L_{ij} \;\; \forall(i,j) \in E \label{eq:asp_l_ncvx_e}
\end{align}
This establishes the non-convex extended DF model.  The convex relaxation, extended CDF, is obtained by relaxing \eqref{eq:asp_l_ncvx_e} to an inequality, 
\begin{align}
& |S_{ij}|^2 \leq \frac{W_i}{|\bm T_{ij}|^2} L_{ij} \;\; \forall(i,j) \in E \label{eq:asp_l_e}
\end{align}
Notice that constraint \eqref{eq:asp_l_e} is a second-order cone constraint as $|\bm T_{ij}|^2$ is a constant.

The complete extended CDF relaxation of \refacepf~is presented in Model \ref{model:cdf_pf_e} \refcdfepf.
The constraints for generator output limits \eqref{eq:ac_2} and line flow limits \eqref{eq:ac_5} are identical to the AC-PF model and the constraints for KCL \eqref{eq:soc_e_1} and \eqref{eq:soc_1} the voltage bounds are identical to the SOC-E-PF model.
Constraints \eqref{eq:cdf_e_3}--\eqref{eq:cdf_e_4} capture line power flow in the $L$-space.  
Constraints \eqref{eq:cdf_e_1}--\eqref{eq:cdf_e_2} capture the voltage angle difference constraints, utilizing \eqref{eq:vv_factor_e}.
Finally constraints \eqref{eq:cdf_e_5} strengthen the relaxation with a second-order cone constraint based on the extended absolute-square of power property.

\begin{model}[t]
\caption{The CDF Relaxation of Extended AC Power Flow (CDF-E-PF)}
\label{model:cdf_pf_e}
\begin{subequations}
\begin{align}
\mbox{\bf variables: \hspace{-0.0cm} } \nonumber \\
& \mbox{Variables from \refcdfpf }  \nonumber \\
\mbox{\bf subject to: \hspace{-0.0cm} } \nonumber \\
& \eqref{eq:ac_2}, \eqref{eq:ac_5}, \eqref{eq:soc_1}, \eqref{eq:soc_e_1}  \nonumber \\ 
& \tan(-\bm {\theta^\Delta}_{ij}) \Re\left( \bm Z^*_{ij}  \bm{T}^*_{ij} \left( \left( \bm Y_{ij} +  \bm {Y^c}_{ij} \right)^* \frac{W_i}{|\bm T_{ij}|^2} - S_{ij} \right) \right) \leq \nonumber \\ 
& \hspace{4.5cm} \Im\left( \bm Z^*_{ij}  \bm{T}^*_{ij} \left( \left( \bm Y_{ij} + \bm {Y^c}_{ij} \right)^* \frac{W_i}{|\bm T_{ij}|^2} - S_{ij} \right)  \right) \;\; \forall (i,j) \in E \label{eq:cdf_e_1}  \\
& \tan(\bm {\theta^\Delta}_{ij}) \Re\left( \bm Z^*_{ij}  \bm{T}^*_{ij} \left( \left( \bm Y_{ij} + \bm {Y^c}_{ij}  \right)^* \frac{W_i}{|\bm T_{ij}|^2} - S_{ij} \right) \right) \geq \nonumber \\ 
& \hspace{4.5cm} \Im\left( \bm Z^*_{ij}  \bm{T}^*_{ij} \left( \left( \bm Y_{ij} + \bm {Y^c}_{ij}  \right)^* \frac{W_i}{|\bm T_{ij}|^2} - S_{ij} \right)  \right) \;\; \forall (i,j) \in E \label{eq:cdf_e_2}  \\
& S_{ij}+S_{ji} = \bm Z_{ij} \left( L_{ij} - \bm {Y^c}_{ij} S_{ij} - \bm {Y^{c*}}_{ij} S^*_{ij} + |\bm {Y^c}_{ij}|^2 \frac{W_i}{|\bm T_{ij}|^2} \right) +  \bm {Y^{c*}}_{ij} \frac{W_i}{|\bm T_{ij}|^2} + \bm {Y^{c*}}_{ji} W_j \;\; \forall (i,j) \in E \label{eq:cdf_e_3}  \\
& \left(1 + \bm Z_{ij}\bm {Y^c}_{ij} + \bm Z^*_{ij}\bm {Y^{c*}}_{ij} \right) \frac{W_i}{|\bm T_{ij}|^2} - W_j = (\bm Z^*_{ij} S_{ij} + \bm Z_{ij} S^*_{ij})   \nonumber \\ 
& \hspace{4.5cm} - |\bm Z_{ij}|^2 \left( L_{ij} - \bm {Y^c}_{ij} S_{ij} - \bm {Y^{c*}}_{ij} S^*_{ij} + |\bm {Y^c}_{ij}|^2 \frac{W_i}{|\bm T_{ij}|^2} \right) \;\; \forall (i,j) \in E \label{eq:cdf_e_4} \\
& |S_{ij}|^2 \leq \frac{W_i}{|\bm T_{ij}|^2} L_{ij}  \;\; \forall (i,j) \in E \label{eq:cdf_e_5} 
\end{align}
\end{subequations}
\end{model}

\subsection{Equivalence of the Relaxations}
\label{sec:equiv}

Despite the significant increase in complexity, the \refsocepf~and \refcdfepf~are equivalent convex relaxations of \refacepf, as demonstrated in this section.

\begin{theorem}
\label{theorem:main}
\refcdfepf~is equivalent to \refsocepf.
\end{theorem}
\begin{proof}

The proof is similar in spirit to those presented in \cite{6483453,6897933} and develops a bijection to establish that any solution of one model can be mapped to a solution of the other model.

We begin by observing that the constraints \eqref{eq:ac_2}, \eqref{eq:ac_5}, \eqref{eq:soc_1}, \eqref{eq:soc_e_1} are included in both models and need not be considered.  
Hence, this result focuses on the voltage angle difference, power flow, and second-order cone constraints in each model.  Noticing that both models share the variables $W_i, S_{ij}, S_{ji}$, the result assumes these values are unchanged when mapping solutions between the to models.

\subsubsection*{\refcdfepf~$\Rightarrow$ \refsocepf}
That is, every solution to \refcdfepf~is a solution to \refsocepf. 
This section will demonstrate that given a solution to the \refcdfepf~$W_i$, $ L_{ij}$, $S_{ij}$, $S_{ji}$ the assignment of the \refsocepf~variables as follows,
\begin{subequations}
\begin{align}
&  W_{ij} =  \bm{T}_{ij} \bm Z^*_{ij} \left( \left( \bm Y_{ij} + \bm {Y^c}_{ij} \right)^* \frac{W_i}{|\bm T_{ij}|^2} - S_{ij}  \right)  \;\; (i,j) \in E \label{eq:soc_cdf}
\end{align}
\end{subequations}
yields a feasible solution to the \refsocepf~model, i.e. all of the model constraints are satisfied.  
The constraints of interest are the voltage angle difference constraints \eqref{eq:soc_2}, the power flow constraints \eqref{eq:soc_e_2}--\eqref{eq:soc_e_3} and the second-order cone constraints \eqref{eq:soc_5}.

The voltage angle difference constraints \eqref{eq:soc_2} are satisfied because mapping function \eqref{eq:soc_cdf} is simply the $W$-space version of the linear relation \eqref{eq:vv_factor_e}, which was originally used to transform \eqref{eq:soc_2} into \eqref{eq:cdf_e_1}--\eqref{eq:cdf_e_2}.

It is easy to see that the power flow constraints \eqref{eq:soc_e_2} are satisfied.  Simply substituting \eqref{eq:soc_cdf} into \eqref{eq:soc_e_2} yields a tautology.  The constraints \eqref{eq:soc_e_3} are more interesting.  First we must observe that the line loss and voltage magnitude difference constraints, \eqref{eq:cdf_e_3}--\eqref{eq:cdf_e_4}, jointly ensure the following property,
\begin{subequations}
\begin{align}
& \bm Z_{ij} S_{ij}^* = \bm Z_{ij}^* S_{ji} + \left(1 + \bm Z_{ij}\bm {Y^c}_{ij} \right) \frac{W_i}{|\bm T_{ij}|^2} - \left(1 + \bm Z^*_{ij}\bm {Y^{c*}}_{ji} \right) W_j  \label{eq:cdf_e_prop}.
\end{align}
\end{subequations}
A complete derivation is presented in Appendix \ref{sec:prop_cdf_e_proof}.  Combing  \eqref{eq:soc_cdf} and \eqref{eq:cdf_e_prop} with the definition of constraint \eqref{eq:soc_e_3} yields a tautology as follows,
\begin{subequations}
\begin{align}
& S_{ji} = \left( \bm Y^*_{ij} + \bm Y^{c*}_{ji} \right) W_j -  \frac{\bm Y^*_{ij}}{\bm{T}^*_{ij}} W^*_{ij} \\
& S_{ji} = \left( \bm Y^*_{ij} + \bm Y^{c*}_{ji} \right) W_j - \frac{\bm Y^*_{ij}}{\bm{T}^*_{ij}} \left( \bm Z_{ij}  \bm{T}^*_{ij} \left( \left( \bm Y_{ij} + \bm {Y^c}_{ij} \right) \frac{W_i}{|\bm T_{ij}|^2} - S^*_{ij} \right) \right) \\
& S_{ji} = \left( \bm Y^*_{ij} + \bm Y^{c*}_{ji} \right) W_j - \bm Y^*_{ij} \frac{W_i}{|\bm{T}_{ij}|^2} - \bm Y_{ij}^* \bm Z_{ij} \bm {Y^c}_{ij} \frac{W_i}{|\bm T_{ij}|^2} + \bm Y_{ij}^* \bm Z_{ij} S^*_{ij} \\
& S_{ji} = \left( \bm Y^*_{ij} + \bm Y^{c*}_{ji} \right) W_j - \left( \bm Y^*_{ij} + \bm Y_{ij}^* \bm Z_{ij} \bm {Y^c}_{ij} \right) \frac{W_i}{|\bm T_{ij}|^2} \nonumber \\ 
& + \bm Y_{ij}^* \left( \bm Z_{ij}^* S_{ji} + \left(1 + \bm Z_{ij}\bm {Y^c}_{ij} \right) \frac{W_i}{|\bm T_{ij}|^2} - \left(1 + \bm Z^*_{ij}\bm {Y^{c*}}_{ji} \right) W_j  \right) \\
& S_{ji} = S_{ji} 
\end{align}
\end{subequations}
Thus demonstrating that the constraints \eqref{eq:soc_e_3} are satisfied.

Finally we consider the second-order cone constraints \eqref{eq:soc_5}.
We begin by developing the value of $|W_{ij}|^2$, given the variable assignment we have selected as follows,
\begin{subequations}
\begin{align}
& |W_{ij}|^2 = \left( \left(1 + \bm Z^*_{ij} \bm Y^{c*}_{ij} \right) \frac{W_i}{\bm T^*_{ij}} - \bm T_{ij} \bm Z^*_{ij} S_{ij} \right)
\left( \left(1 + \bm Z_{ij} \bm Y^{c}_{ij} \right) \frac{W_i}{\bm T_{ij}} - \bm T^*_{ij} \bm Z_{ij} S^*_{ij} \right) \\
&  |W_{ij}|^2 = \left(1 + \bm Z_{ij} \bm Y^{c}_{ij} + \bm Z^*_{ij} \bm Y^{c*}_{ij} + |\bm Z_{ij}|^2 |\bm Y^{c}_{ij}|^2 \right) \frac{|W_i|^2}{|\bm T_{ij}|^2} \nonumber \\
& -  \left(1 + \bm Z^*_{ij} \bm Y^{c*}_{ij} \right) \bm Z_{ij} S^*_{ij} W_i 
  - \left(1 + \bm Z_{ij} \bm Y^{c}_{ij} \right) \bm Z^*_{ij} S_{ij} W_i
  + |\bm T_{ij}|^2 |\bm Z_{ij}|^2 |S_{ij}|^2 \\
&  |W_{ij}|^2 = |\bm Z_{ij}|^2 \left( |\bm T_{ij}|^2 |S_{ij}|^2 - \bm Y^{c*}_{ij} S_{ij} W_i - \bm Y^{c}_{ij}S^*_{ij} W_i + |\bm Y^{c*}_{ij}|^2 \frac{W_i^2}{|\bm T_{ij}|^2} \right) \nonumber \\
& - \bm Z_{ij} S^*_{ij} W_i - \bm Z^*_{ij} S_{ij} W_i + \left(1 + \bm Z_{ij} \bm Y^{c}_{ij} + \bm Z^*_{ij} \bm Y^{c*}_{ij} \right)  \frac{W_i^2}{|\bm T_{ij}|^2} \\
&  |W_{ij}|^2 = W_i |\bm Z_{ij}|^2 \left( \frac{|\bm T_{ij}|^2 |S_{ij}|^2}{W_i} - \bm Y^{c}_{ij}S^*_{ij} - \bm Y^{c*}_{ij} S_{ij} + |\bm Y^{c*}_{ij}|^2 \frac{W_i}{|\bm T_{ij}|^2} \right) \nonumber \\
& + W_i \left( - \bm Z_{ij} S^*_{ij} - \bm Z^*_{ij} S_{ij} + \left(1 + \bm Z_{ij} \bm Y^{c}_{ij} + \bm Z^*_{ij} \bm Y^{c*}_{ij} \right)  \frac{W_i}{|\bm T_{ij}|^2} \right) \\
&  |W_{ij}|^2 \leq W_i |\bm Z_{ij}|^2 \left( L_{ij} - \bm Y^{c}_{ij}S^*_{ij} - \bm Y^{c*}_{ij} S_{ij} + |\bm Y^{c*}_{ij}|^2 \frac{W_i}{|\bm T_{ij}|^2} \right) \nonumber \\
& + W_i \left( - \bm Z_{ij} S^*_{ij} - \bm Z^*_{ij} S_{ij} + \left(1 + \bm Z_{ij} \bm Y^{c}_{ij} + \bm Z^*_{ij} \bm Y^{c*}_{ij} \right)  \frac{W_i}{|\bm T_{ij}|^2} \right)
\end{align}
\end{subequations}
From this point, we apply the second-order cone constraint from \refcdfepf, i.e.~\eqref{eq:cdf_e_5}, to eliminate the $|S_{ij}|^2$ variable and then apply the voltage magnitude difference equation \eqref{eq:cdf_e_4} to simplify the large expression to $W_j$,
\begin{subequations}
\begin{align}
 %
 &  |W_{ij}|^2 \leq W_i \left( |\bm Z_{ij}|^2 \left( L_{ij} - \bm Y^{c}_{ij}S^*_{ij} - \bm Y^{c*}_{ij} S_{ij} + |\bm Y^{c*}_{ij}|^2 \frac{W_i}{|\bm T_{ij}|^2} \right) - \bm Z_{ij} S^*_{ij} - \bm Z^*_{ij} S_{ij} + \left(1 + \bm Z_{ij} \bm Y^{c}_{ij} + \bm Z^*_{ij} \bm Y^{c*}_{ij} \right)  \frac{W_i}{|\bm T_{ij}|^2} \right) \\
&  |W_{ij}|^2 \leq W_i W_j
\end{align}
\end{subequations}
Demonstrating that constraint \eqref{eq:soc_5} is satisfied in \refsocepf.
With all of the constraints in \refsocepf~satisfied, the proof continues by performing a similar analysis in the reverse direction.

\subsubsection*{\refsocepf~$\Rightarrow$ \refcdfepf}
That is, every solution to \refsocepf~is a solution to \refcdfepf.  This section will demonstrate that given a solution to the \refsocepf~$W_i, W_{ij}, S_{ij}, S_{ji}$ the assignment of the \refcdfepf~variables,
\begin{subequations}
\begin{align}
& L_{ij} = |\bm Y_{ij}|^2 \left( \frac{W_i}{|\bm T_{ij}|^2} - \frac{W_{ij}}{\bm T_{ij}} - \frac{W_{ij}^*}{\bm T^*_{ij}} + W_j  \right) 
  + \bm {Y^c}_{ij} S_{ij} +   \bm {Y^{c*}}_{ij} S^*_{ij} 
  - |\bm {Y^c}_{ij}|^2 \frac{|V_i|^2}{|\bm T_{ij}|^2}  \;\; \forall(i,j) \in E   \label{eq:socp_cdf_l}
\end{align}
\end{subequations}
produces a feasible solution to the \refcdfepf~model, i.e. all of the model constraints are satisfied.  
The constraints of interest are the voltage angle difference constraints \eqref{eq:cdf_e_1}--\eqref{eq:cdf_e_2}, the line loss and voltage magnitude difference constraints \eqref{eq:cdf_e_3}--\eqref{eq:cdf_e_4} and second-order cone constraints \eqref{eq:cdf_e_5}.

The voltage angle difference constraints \eqref{eq:cdf_e_1}--\eqref{eq:cdf_e_2} are stratified because the constraint \eqref{eq:soc_e_2} is simply the $W$-space version of the linear relation \eqref{eq:vv_factor_e}, which was originally used to transform \eqref{eq:soc_2} in to \eqref{eq:cdf_e_1}--\eqref{eq:cdf_e_2}.

Continuing with line loss constraint \eqref{eq:cdf_e_3}, we simply expand the value of $L_{ij}$ as defined in \eqref{eq:socp_cdf_l} as follows,
\begin{subequations}
\begin{align}
& S_{ij}+S_{ji} = \bm Z_{ij} \left( L_{ij} - \bm {Y^c}_{ij} S_{ij} - \bm {Y^{c*}}_{ij} S^*_{ij} + |\bm {Y^c}_{ij}|^2 \frac{W_i}{|\bm T_{ij}|^2}  \right) +  \bm {Y^{c*}}_{ij} \frac{W_i}{|\bm T_{ij}|^2} + \bm {Y^{c*}}_{ji} W_j \\
&  S_{ij} +  S_{ji} = 
\bm Y^*_{ij} \left( \frac{W_i}{ |\bm{T}_{ij}|^2} - \frac{W_{ij}}{\bm{T}_{ij}} - \frac{W^*_{ij}}{\bm{T}^*_{ij}} + W_j \right)
+  \bm {Y^{c*}}_{ij} \frac{W_i}{|\bm T_{ij}|^2} + \bm {Y^{c*}}_{ji} W_j 
\end{align}
\end{subequations}
Noticing that this is the extended line loss property \eqref{eq:v_loss_e} in the $W$-space demonstrates that the constraints hold.
Next, the voltage magnitude difference constraint \eqref{eq:cdf_e_4} is expanded in a similar way,
\begin{subequations}
\begin{align}
& \left(1 + \bm Z_{ij}\bm {Y^c}_{ij} + \bm Z^*_{ij}\bm {Y^{c*}}_{ij} \right) \frac{W_i}{|\bm T_{ij}|^2} - W_j = \bm Z^*_{ij} S_{ij} + \bm Z_{ij} S^*_{ij}  - |\bm Z_{ij}|^2 \left( L_{ij} + |\bm {Y^c}_{ij}|^2 \frac{W_i}{|\bm T_{ij}|^2} - \bm {Y^c}_{ij} S_{ij} - \bm {Y^{c*}}_{ij} S^*_{ij} \right)  \\
& \left(1 + \bm Z_{ij}\bm {Y^c}_{ij} + \bm Z^*_{ij}\bm {Y^{c*}}_{ij} \right) \frac{W_i}{|\bm T_{ij}|^2} - W_j = \bm Z^*_{ij} S_{ij} + \bm Z_{ij} S^*_{ij} -  |\bm Z_{ij}|^2 \left(  \frac{W_i}{ |\bm{T}_{ij}|^2} - \frac{W_{ij}}{\bm{T}_{ij}} - \frac{W^*_{ij}}{\bm{T}^*_{ij}} + W_j \right)
\end{align}
\end{subequations}
Noticing that this is simply the extended voltage magnitude difference property \eqref{eq:v_drop_e} in the $W$-space demonstrates that these constraints hold.

Finally we consider the second-order cone constraints \eqref{eq:cdf_e_5}.
We begin by developing the value of $|S_{ij}|^2$, given the model constraints as follows,
\begin{subequations}
\begin{align}
& |S_{ij}|^2 = 
 \left( \left( \bm Y_{ij} + \bm {Y^c}_{ij} \right)^* \frac{W_i}{|\bm{T}_{ij}|^2} - \bm Y^*_{ij} \frac{W_{ij}}{\bm{T}_{ij}} \right)
 \left( \left( \bm Y_{ij} + \bm {Y^c}_{ij} \right) \frac{W^*_i}{|\bm{T}_{ij}|^2} - \bm Y_{ij} \frac{W^*_{ij}}{\bm{T}^*_{ij}} \right) \\
& |S_{ij}|^2 = \left( \bm Y^*_{ij} + \bm {Y^{c*}}_{ij} \right) \left( \bm Y_{ij} + \bm {Y^c}_{ij} \right) \frac{|W_i|^2}{|\bm{T}_{ij}|^4}
- \bm Y_{ij} \left( \bm Y^*_{ij} + \bm {Y^{c*}}_{ij} \right) \frac{W_i}{|\bm{T}_{ij}|^2} \frac{W^*_{ij}}{\bm{T}^*_{ij}}
- \bm Y^*_{ij} \left( \bm Y_{ij} + \bm {Y^{c}}_{ij} \right) \frac{W^*_i}{|\bm{T}_{ij}|^2} \frac{W_{ij}}{\bm{T}_{ij}} \nonumber \\
& + |\bm Y_{ij}|^2 \frac{|W_{ij}|^2}{|\bm{T}_{ij}|^2} \\
& |S_{ij}|^2 =  |\bm Y_{ij}|^2 \left( \frac{|W_i|^2}{|\bm{T}_{ij}|^4} - \frac{W_i}{|\bm{T}_{ij}|^2} \frac{W^*_{ij}}{\bm{T}^*_{ij}} - \frac{W^*_i}{|\bm{T}_{ij}|^2} \frac{W_{ij}}{\bm{T}_{ij}} + \frac{|W_{ij}|^2}{|\bm{T}_{ij}|^2} \right) \nonumber \\
& + \bm Y_{ij} \bm {Y^{c*}}_{ij} \frac{|W_i|^2}{|\bm{T}_{ij}|^4} - \bm Y_{ij} \bm {Y^{c*}}_{ij} \frac{W_i}{|\bm{T}_{ij}|^2} \frac{W^*_{ij}}{\bm{T}^*_{ij}}
+ \bm Y^*_{ij} \bm {Y^{c}}_{ij}  \frac{|W_i|^2}{|\bm{T}_{ij}|^4} - \bm Y^*_{ij} \bm {Y^{c}}_{ij}  \frac{W^*_i}{|\bm{T}_{ij}|^2} \frac{W_{ij}}{\bm{T}_{ij}}
+ |\bm {Y^c}_{ij}|^2 \frac{|W_{i}|^2}{|\bm{T}_{ij}|^4}
\end{align}
\end{subequations}
At this point the expression $|\bm {Y^c}_{ij}|^2 |W_{i}|^2 / |\bm{T}_{ij}|^4 - |\bm {Y^c}_{ij}|^2 |W_{i}|^2 / |\bm{T}_{ij}|^4$ is introduced so equation \eqref{eq:soc_e_2} can be leveraged to introduce the $S_{ij}$ variables as follows,
\begin{subequations}
\begin{align}
& |S_{ij}|^2 =  |\bm Y_{ij}|^2 \left( \frac{|W_i|^2}{|\bm{T}_{ij}|^4} - \frac{W_i}{|\bm{T}_{ij}|^2} \frac{W^*_{ij}}{\bm{T}^*_{ij}} - \frac{W^*_i}{|\bm{T}_{ij}|^2} \frac{W_{ij}}{\bm{T}_{ij}} + \frac{|W_{ij}|^2}{|\bm{T}_{ij}|^2} \right) \nonumber \\
& + \bm {Y^{c*}}_{ij} \frac{W_i}{|\bm{T}_{ij}|^2}  \left( ( \bm Y_{ij} + \bm {Y^{c}}_{ij}) \frac{W^*_i}{|\bm{T}_{ij}|^2} - \bm Y_{ij} \frac{W^*_{ij}}{\bm{T}^*_{ij}} \right)
+ \bm {Y^{c}}_{ij}  \frac{W_i}{|\bm{T}_{ij}|^2} \left( ( \bm Y^*_{ij} + \bm {Y^{c*}}_{ij}) \frac{W_i}{|\bm{T}_{ij}|^2} - \bm Y^*_{ij} \frac{W_{ij}}{\bm{T}_{ij}} \right) \nonumber \\
& - |\bm {Y^c}_{ij}|^2 \frac{|W_{i}|^2}{|\bm{T}_{ij}|^4} \\
& |S_{ij}|^2 =  |\bm Y_{ij}|^2 \left( \frac{|W_i|^2}{|\bm{T}_{ij}|^4} - \frac{W_i}{|\bm{T}_{ij}|^2} \frac{W^*_{ij}}{\bm{T}^*_{ij}} - \frac{W_i}{|\bm{T}_{ij}|^2} \frac{W_{ij}}{\bm{T}_{ij}} + \frac{|W_{ij}|^2}{|\bm{T}_{ij}|^2} \right) 
+ \bm {Y^{c*}}_{ij} \frac{W_i}{|\bm{T}_{ij}|^2} S^*_{ij} + \bm {Y^{c}}_{ij}  \frac{W^*_i}{|\bm{T}_{ij}|^2} S_{ij} - |\bm {Y^c}_{ij}|^2 \frac{|W_{i}|^2}{|\bm{T}_{ij}|^4} 
\end{align}
\end{subequations}
From this point a factor of $W_i / |\bm{T}_{ij}|^2$ can be collected from all of the terms.  However, to extract this term from $|W_{ij}|^2 / |\bm{T}_{ij}|^2$ we apply the SOC constraint from \refsocepf, i.e.~\eqref{eq:soc_5}, which turns the equation into an inequality.  Lastly the assignment of $L_{ij}$ \eqref{eq:socp_cdf_l} is used to complete the derivation as follows,
\begin{subequations}
\begin{align}
& |S_{ij}|^2 \leq \frac{W_i}{|\bm{T}_{ij}|^2} \left( |\bm Y_{ij}|^2 \left( \frac{W_i}{|\bm{T}_{ij}|^2} - \frac{W^*_{ij}}{\bm{T}^*_{ij}} - \frac{W_{ij}}{\bm{T}_{ij}} + W_j \right) 
+ ( \bm {Y^{c*}}_{ij} S^*_{ij} + \bm {Y^{c}}_{ij} S_{ij} ) - |\bm {Y^c}_{ij}|^2 \frac{W_{i}}{|\bm{T}_{ij}|^2} \right) \\
&  |S_{ij}|^2 \leq \frac{W_i}{ |\bm{T}_{ij}|^2} L_{ij}
\end{align}
\end{subequations}
Demonstrating that constraint \eqref{eq:cdf_e_5} is satisfied in \refcdfepf, and completing the proof that these two extended relaxations are equivalent.
\end{proof}

\section{Conclusion}
\label{sec:conc}

This paper increases the applicability of the established Convex-DistFlow (CDF) relaxation of the power flow equations by proposing an extended CDF model,  which is applicable to industrial transmission system datasets.  Additionally, it was shown that this extended CDF model defines the same convex set as the well known SOC relaxation for transmission systems.
Given that these relaxations define the same set of solutions, the natural frontier for future work is to conduct a detailed numerical study of these two relaxations.
Although the relaxation quality may be identical, significant differences in the formulations may lead to variations in the computation time of the numerical methods used to solve these equations.

\section{Acknowledgements}

This work was partially funded by NICTA, which was supported by the Australian Government through the Department of Communications and the Australian Research Council through the ICT Centre of Excellence Program.

\bibliographystyle{plain}
\bibliography{power_models}

\begin{thebibliography}{10}

\bibitem{Bai2008383}
Xiaoqing Bai, Hua Wei, Katsuki Fujisawa, and Yong Wang.
\newblock Semidefinite programming for optimal power flow problems.
\newblock {\em International Journal of Electrical Power \& Energy Systems},
  30(6–7):383 -- 392, 2008.

\bibitem{19266}
M.E. Baran and F.F. Wu.
\newblock Optimal sizing of capacitors placed on a radial distribution system.
\newblock {\em IEEE Transactions on Power Delivery}, 4(1):735--743, Jan 1989.

\bibitem{6897933}
S.~Bose, S.H. Low, T.~Teeraratkul, and B.~Hassibi.
\newblock Equivalent relaxations of optimal power flow.
\newblock {\em IEEE Transactions on Automatic Control}, PP(99):1--1, 2014.

\bibitem{LPAC_ijoc}
Coffrin Carleton and Pascal Van~Hentenryck.
\newblock A linear-programming approximation of ac power flows.
\newblock {\em Forthcoming in INFORMS Journal on Computing}, 2014.

\bibitem{7763860}
C.~Coffrin, H.~L. Hijazi, and P.~Van Hentenryck.
\newblock Strengthening the sdp relaxation of ac power flows with convex
  envelopes, bound tightening, and valid inequalities.
\newblock {\em IEEE Transactions on Power Systems}, 32(5):3549--3558, Sept
  2017.

\bibitem{qc_opf_tps}
C.~Coffrin, H.L. Hijazi, and P.~Van~Hentenryck.
\newblock The qc relaxation: A theoretical and computational study on optimal
  power flow.
\newblock {\em IEEE Transactions on Power Systems}, PP(99):1--11, 2015.

\bibitem{nesta}
Carleton Coffrin, Dan Gordon, and Paul Scott.
\newblock {NESTA, The {\sc Nicta} Energy System Test Case Archive}.
\newblock {\em CoRR}, abs/1411.0359, 2014.

\bibitem{6102366}
M.~Farivar, C.R. Clarke, S.H. Low, and K.M. Chandy.
\newblock Inverter var control for distribution systems with renewables.
\newblock In {\em 2011 IEEE International Conference on Smart Grid
  Communications (SmartGridComm)}, pages 457--462, Oct 2011.

\bibitem{6507355}
M.~Farivar and S.H. Low.
\newblock {Branch Flow Model: Relaxations and Convexification, Part I}.
\newblock {\em IEEE Transactions on Power Systems}, 28(3):2554--2564, Aug 2013.

\bibitem{6507352}
M.~Farivar and S.H. Low.
\newblock {Branch Flow Model: Relaxations and Convexification, Part II}.
\newblock {\em IEEE Transactions on Power Systems}, 28(3):2565--2572, Aug 2013.

\bibitem{gurobi}
{Gurobi Optimization, Inc.}
\newblock Gurobi optimizer reference manual.
\newblock Published online at \url{http://www.gurobi.com}, 2014.

\bibitem{Hijazi2017}
Hassan Hijazi, Carleton Coffrin, and Pascal~Van Hentenryck.
\newblock Convex quadratic relaxations for mixed-integer nonlinear programs in
  power systems.
\newblock {\em Mathematical Programming Computation}, 9(3):321--367, Sep 2017.

\bibitem{cplex}
Inc. IBM.
\newblock {IBM ILOG CPLEX Optimization Studio}.
\newblock
  \url{http://www-01.ibm.com/software/commerce/optimization/cplex-optimizer/},
  2014.

\bibitem{Jabr06}
R.A. Jabr.
\newblock Radial distribution load flow using conic programming.
\newblock {\em IEEE Transactions on Power Systems}, 21(3):1458--1459, Aug 2006.

\bibitem{9780070359581}
Prabha Kundur.
\newblock {\em Power System Stability and Control}.
\newblock McGraw-Hill Professional, 1994.

\bibitem{5971792}
J.~Lavaei and S.H. Low.
\newblock Zero duality gap in optimal power flow problem.
\newblock {\em IEEE Transactions on Power Systems}, 27(1):92 --107, feb. 2012.

\bibitem{ACSTAR2015}
Karsten Lehmann, Alban Grastien, and Pascal Van~Hentenryck.
\newblock {AC-Feasibility on Tree Networks is NP-Hard}.
\newblock {\em IEEE Transactions on Power Systems}, 2015 (to appear).

\bibitem{6756976}
S.H. Low.
\newblock Convex relaxation of optimal power flow - part i: Formulations and
  equivalence.
\newblock {\em IEEE Transactions on Control of Network Systems}, 1(1):15--27,
  March 2014.

\bibitem{6815671}
S.H. Low.
\newblock Convex relaxation of optimal power flow - part ii: Exactness.
\newblock {\em IEEE Transactions on Control of Network Systems}, 1(2):177--189,
  June 2014.

\bibitem{6822653}
R.~Madani, S.~Sojoudi, and J.~Lavaei.
\newblock Convex relaxation for optimal power flow problem: Mesh networks.
\newblock {\em IEEE Transactions on Power Systems}, 30(1):199--211, Jan 2015.

\bibitem{7038397}
D.K. Molzahn and I.A. Hiskens.
\newblock Moment-based relaxation of the optimal power flow problem.
\newblock In {\em Power Systems Computation Conference (PSCC), 2014}, pages
  1--7, Aug 2014.

\bibitem{6980142}
D.K. Molzahn and I.A. Hiskens.
\newblock Sparsity-exploiting moment-based relaxations of the optimal power
  flow problem.
\newblock {\em Power Systems, IEEE Transactions on}, PP(99):1--13, 2014.

\bibitem{Purchala:2005gt}
K~Purchala, L~Meeus, D~Van~Dommelen, and R~Belmans.
\newblock {Usefulness of {DC} power flow for active power flow analysis}.
\newblock {\em Power Engineering Society General Meeting}, pages 454--459,
  2005.

\bibitem{6345272}
S.~Sojoudi and J.~Lavaei.
\newblock Physics of power networks makes hard optimization problems easy to
  solve.
\newblock In {\em Power and Energy Society General Meeting, 2012 IEEE}, pages
  1--8, July 2012.

\bibitem{6483453}
B.~Subhonmesh, S.H. Low, and K.M. Chandy.
\newblock Equivalence of branch flow and bus injection models.
\newblock In {\em Communication, Control, and Computing (Allerton), 2012 50th
  Annual Allerton Conference on}, pages 1893--1899, Oct 2012.

\bibitem{pglib_opf}
{The IEEE PES Task Force on Benchmarks for Validation of Emerging Power System
  Algorithms}.
\newblock {PGLib Optimal Power Flow Benchmarks}.
\newblock Published online at
  \url{https://github.com/power-grid-lib/pglib-opf}.
\newblock Accessed: October 4, 2017.

\bibitem{mosek}
K.~C. Toh, R.~H. TŸtŸncŸ, and M.~J. Todd.
\newblock {SDPT3 - a MATLAB software package for semidefinite-quadratic-linear
  programming}.
\newblock \url{https://mosek.com/}, 2014.

\bibitem{verma2009power}
Abhinav Verma.
\newblock {\em Power grid security analysis: An optimization approach}.
\newblock PhD thesis, Columbia University, 2009.

\bibitem{matpower}
R.D. Zimmerman, C.E. Murillo-S‡~andnchez, and R.J. Thomas.
\newblock Matpower: Steady-state operations, planning, and analysis tools for
  power systems research and education.
\newblock {\em IEEE Transactions on Power Systems}, 26(1):12 --19, feb. 2011.

\end{thebibliography}

LA-UR-17-25214

\clearpage
\appendix

\section{Complete Derivation of the Extended Properties}
\label{sec:prop_e_proof}

This appendix provides detailed derivations of a variety of general power flow properties in the extended AC-PF model, which are proposed in Section \ref{sec:ac_prop_e}.
Note that these derivations make use of the following properties of complex numbers,
\begin{subequations}
\begin{align}
& X + X^* = 2\Re(X) = 2\Re(X^*) \nonumber \\
& X - X^* = \bm i 2\Im(X) = -\bm i 2\Im(X^*)  \nonumber 
\end{align}
\end{subequations}
and leverage \eqref{eq:line_s_e_from}-\eqref{eq:line_s_e_to} to simplify the presentation using $S_{ij}$.

\paragraph{Extended Absolute Square of Current}
\begin{subequations}
\begin{align}
& I_{ij}I_{ij}^*  = \left(  \left( \bm Y_{ij} + \bm {Y^c}_{ij} \right) \frac{V_i}{\bm T_{ij}} - \bm Y_{ij} V_j  \right) \left(  \left( \bm Y_{ij} + \bm {Y^c}_{ij} \right)^* \frac{V_i^*}{\bm T^*_{ij}} - \bm Y_{ij}^* V_j^*  \right)  \\
& I_{ij}I_{ij}^*  = \left( \bm Y_{ij} + \bm {Y^c}_{ij} \right) \left( \bm Y_{ij} + \bm {Y^c}_{ij} \right)^* \frac{|V_i|^2}{|\bm T_{ij}|^2} - \bm Y^*_{ij} (\bm Y_{ij} + \bm {Y^c}_{ij}) \frac{V_i V_j^*}{\bm T_{ij}} - \bm Y_{ij} (\bm Y_{ij} + \bm {Y^c}_{ij})^* \frac{V^*_i V_j}{\bm T^*_{ij}} +  |\bm Y_{ij}|^2 |V_j|^2  \\
& I_{ij}I_{ij}^*  = |\bm Y_{ij}|^2 \left( \frac{|V_i|^2}{|\bm T_{ij}|^2} - \frac{V_i V^*_j}{\bm T_{ij}} - \frac{V^*_i V_j}{\bm T^*_{ij}} - |V_j|^2 \right) \nonumber \\
& + |\bm {Y^c}_{ij}|^2 \frac{|V_i|^2}{|\bm T_{ij}|^2} + \bm {Y^c}_{ij} Y^*_{ij} \frac{|V_i|^2}{|\bm T_{ij}|^2} - \bm {Y^c}_{ij} \bm Y^*_{ij} \frac{V_i V^*_j}{\bm T_{ij}}
+ \bm {Y^{c*}}_{ij} Y_{ij} \frac{|V_i|^2}{|\bm T_{ij}|^2} - \bm {Y^{c*}}_{ij} \bm Y_{ij} \frac{V^*_i V_j}{\bm T^*_{ij}}
\end{align}
Now $\bm {Y^c}_{ij} |V_i|^2/|\bm{T}_{ij}|^2 - \bm {Y^c}_{ij} |V_i|^2/|\bm{T}_{ij}|^2$ is introduced to generate $S_{ij}$ terms as follows,
\begin{align}
& I_{ij}I_{ij}^*  = |\bm Y_{ij}|^2 \left( \frac{|V_i|^2}{|\bm T_{ij}|^2} - \frac{V_i V^*_j}{\bm T_{ij}} - \frac{V^*_i V_j}{\bm T^*_{ij}} - |V_j|^2 \right) \nonumber \\
& + |\bm {Y^c}_{ij}|^2 \frac{|V_i|^2}{|\bm T_{ij}|^2} + \bm {Y^c}_{ij} Y^*_{ij} \frac{|V_i|^2}{|\bm T_{ij}|^2} - \bm {Y^c}_{ij} \bm Y^*_{ij} \frac{V_i V^*_j}{\bm T_{ij}}
+ |\bm {Y^c}_{ij}|^2 \frac{|V_i|^2}{|\bm T_{ij}|^2}  + \bm {Y^{c*}}_{ij} Y_{ij} \frac{|V_i|^2}{|\bm T_{ij}|^2} - \bm {Y^{c*}}_{ij} \bm Y_{ij} \frac{V^*_i V_j}{\bm T^*_{ij}} \nonumber \\
& - |\bm {Y^c}_{ij}|^2 \frac{|V_i|^2}{|\bm T_{ij}|^2} \\
& I_{ij}I_{ij}^*  = |\bm Y_{ij}|^2 \left( \frac{|V_i|^2}{|\bm T_{ij}|^2} - \frac{V_i V^*_j}{\bm T_{ij}} - \frac{V^*_i V_j}{\bm T^*_{ij}} - |V_j|^2 \right) \nonumber \\
& + \bm {Y^c}_{ij} \left( \bm {Y^{c*}}_{ij} \frac{|V_i|^2}{|\bm T_{ij}|^2} +  Y^*_{ij} \frac{|V_i|^2}{|\bm T_{ij}|^2} - \bm Y^*_{ij} \frac{V_i V^*_j}{\bm T_{ij}} \right)
+ \bm {Y^{c*}}_{ij} \left( \bm {Y^c}_{ij} \frac{|V_i|^2}{|\bm T_{ij}|^2}  + Y_{ij} \frac{|V_i|^2}{|\bm T_{ij}|^2} - \bm Y_{ij} \frac{V^*_i V_j}{\bm T^*_{ij}}  \right) \nonumber \\
& - |\bm {Y^c}_{ij}|^2 \frac{|V_i|^2}{|\bm T_{ij}|^2} \\
& I_{ij}I_{ij}^*  = |\bm Y_{ij}|^2 \left( \frac{|V_i|^2}{|\bm T_{ij}|^2} - \frac{V_i V^*_j}{\bm T_{ij}} - \frac{V^*_i V_j}{\bm T^*_{ij}} - |V_j|^2 \right) + \bm {Y^c}_{ij} S_{ij} + \bm {Y^{c*}}_{ij} S^*_{ij}  - |\bm {Y^c}_{ij}|^2 \frac{|V_i|^2}{|\bm T_{ij}|^2} \\
\end{align}
\end{subequations}

\paragraph{Extended Voltage Magnitude Difference}
\begin{subequations}
\begin{align}
& S_{ij} - S_{ji}  = \bm Y^*_{ij} \left( \frac{|V_i|^2}{|\bm{T}_{ij}|^2} - |V_j|^2 \right) - \bm Y^*_{ij} \left( \frac{V_iV_j^*}{\bm{T}_{ij}} - \frac{V_i^*V_j}{\bm{T}^*_{ij}} \right) + \bm {Y^{c*}}_{ij} \frac{|V_i|^2}{|\bm{T}_{ij}|^2} - \bm {Y^{c*}}_{ji} |V_j|^2 \\
& \bm Z^*_{ij} S_{ij} - \bm Z^*_{ij} S_{ji}  = \frac{|V_i|^2}{|\bm{T}_{ij}|^2} - |V_j|^2 - \frac{V_iV_j^*}{\bm{T}_{ij}} + \frac{V_i^*V_j}{\bm{T}^*_{ij}} + \bm Z^*_{ij} \bm {Y^{c*}}_{ij} \frac{|V_i|^2}{|\bm{T}_{ij}|^2} - \bm Z^*_{ij} \bm {Y^{c*}}_{ji} |V_j|^2 \\
& ( 1+ \bm Z^*_{ij} \bm {Y^{c*}}_{ij} ) \frac{|V_i|^2}{|\bm{T}_{ij}|^2} - (1 + \bm Z^*_{ij} \bm {Y^{c*}}_{ji} )|V_j|^2 = \bm Z^*_{ij} S_{ij} - \bm Z^*_{ij} S_{ji} + \frac{V_iV_j^*}{\bm{T}_{ij}} - \frac{V_i^*V_j}{\bm{T}^*_{ij}}
\end{align}
Then expand $S_{ji}$ using \eqref{eq:v_loss_e} and $S_{ij}$ using \eqref{eq:line_s_e_from} as follows,
\begin{align} 
& ( 1+ \bm Z^*_{ij} \bm {Y^{c*}}_{ij} ) \frac{|V_i|^2}{|\bm{T}_{ij}|^2} - (1 + \bm Z^*_{ij} \bm {Y^{c*}}_{ji} )|V_j|^2 = \frac{V_iV_j^*}{\bm{T}_{ij}} - \frac{V_i^*V_j}{\bm{T}^*_{ij}} + \nonumber \\
& \bm Z^*_{ij} S_{ij} - \bm Z^*_{ij} \left( \bm Y^*_{ij} \left( \frac{|V_i|^2}{|\bm T_{ij}|^2} - \frac{V_iV_j^*}{\bm T_{ij}} - \frac{V_i^*V_j}{\bm T^*_{ij}} + |V_j|^2  \right) +  \bm {Y^{c*}}_{ij} \frac{|V_i|^2}{|\bm T_{ij}|^2} + \bm {Y^{c*}}_{ji} |V_j|^2  - S_{ij} \right) \\
& ( 1+ \bm Z^*_{ij} \bm {Y^{c*}}_{ij} ) \frac{|V_i|^2}{|\bm{T}_{ij}|^2} - (1 + \bm Z^*_{ij} \bm {Y^{c*}}_{ji} )|V_j|^2 = \frac{V_iV_j^*}{\bm{T}_{ij}} - \frac{V_i^*V_j}{\bm{T}^*_{ij}} + \nonumber \\
& \bm Z^*_{ij} S_{ij} - \left( \frac{|V_i|^2}{|\bm T_{ij}|^2} - \frac{V_iV_j^*}{\bm T_{ij}} - \frac{V_i^*V_j}{\bm T^*_{ij}} + |V_j|^2  \right) -   \bm Z^*_{ij} \bm {Y^{c*}}_{ij} \frac{|V_i|^2}{|\bm T_{ij}|^2} -  \bm Z^*_{ij} \bm {Y^{c*}}_{ji} |V_j|^2  + \bm Z^*_{ij}S_{ij}  \\ \\
& ( 1+ \bm Z^*_{ij} \bm {Y^{c*}}_{ij} ) \frac{|V_i|^2}{|\bm{T}_{ij}|^2} - (1 + \bm Z^*_{ij} \bm {Y^{c*}}_{ji} )|V_j|^2 = \frac{V_iV_j^*}{\bm{T}_{ij}} - \frac{V_i^*V_j}{\bm{T}^*_{ij}} + \nonumber \\
& \bm Z^*_{ij} S_{ij} - \left( \frac{|V_i|^2}{|\bm T_{ij}|^2} - \frac{V_iV_j^*}{\bm T_{ij}} - \frac{V_i^*V_j}{\bm T^*_{ij}} + |V_j|^2  \right) -   \bm Z^*_{ij} \bm {Y^{c*}}_{ij} \frac{|V_i|^2}{|\bm T_{ij}|^2} -  \bm Z^*_{ij} \bm {Y^{c*}}_{ji} |V_j|^2  +  \nonumber \\
& \bm Z^*_{ij} \left( \left( \bm Y_{ij} + \bm {Y^c}_{ij} \right)^* \frac{|V_i|^2}{|\bm{T}_{ij}|^2} - \bm Y^*_{ij} \frac{V_iV^*_j}{\bm{T}_{ij}} \right) \\
& ( 1+ \bm Z^*_{ij} \bm {Y^{c*}}_{ij} ) \frac{|V_i|^2}{|\bm{T}_{ij}|^2} - |V_j|^2 =  \bm Z^*_{ij} S_{ij} + \frac{|V_i|^2}{|\bm{T}_{ij}|^2} - \frac{V_i^*V_j}{\bm{T}^*_{ij}} - \left( \frac{|V_i|^2}{|\bm T_{ij}|^2} - \frac{V_iV_j^*}{\bm T_{ij}} - \frac{V_i^*V_j}{\bm T^*_{ij}} + |V_j|^2  \right) 
\end{align}
Finally $\bm Z_{ij} \bm {Y^c}_{ij} |V_i|^2/|\bm{T}_{ij}|^2 - \bm Z_{ij} \bm {Y^c}_{ij} |V_i|^2/|\bm{T}_{ij}|^2$ is introduced to generate a $S^*_{ij}$ term as follows,
\begin{align} 
& ( 1+ \bm Z^*_{ij} \bm {Y^{c*}}_{ij} + \bm Z_{ij} \bm {Y^c}_{ij} ) \frac{|V_i|^2}{|\bm{T}_{ij}|^2} - |V_j|^2 =  \bm Z^*_{ij} S_{ij} + (1+\bm Z_{ij} \bm {Y^c}_{ij}) \frac{|V_i|^2}{|\bm{T}_{ij}|^2} - \frac{V_i^*V_j}{\bm{T}^*_{ij}} - \nonumber \\ 
& \left( \frac{|V_i|^2}{|\bm T_{ij}|^2} - \frac{V_iV_j^*}{\bm T_{ij}} - \frac{V_i^*V_j}{\bm T^*_{ij}} + |V_j|^2  \right) \\
& ( 1+ \bm Z^*_{ij} \bm {Y^{c*}}_{ij} + \bm Z_{ij} \bm {Y^c}_{ij} ) \frac{|V_i|^2}{|\bm{T}_{ij}|^2} - |V_j|^2 = \bm Z^*_{ij} S_{ij} +  \bm Z_{ij} S^*_{ij} - \left( \frac{|V_i|^2}{|\bm T_{ij}|^2} - \frac{V_iV_j^*}{\bm T_{ij}} - \frac{V_i^*V_j}{\bm T^*_{ij}} + |V_j|^2  \right) 
\end{align}
\end{subequations}

\section{Combination of Line Loss and Voltage Magnitude Difference}
\label{sec:prop_cdf_e_proof}

This appendix provides a detailed derivation of the $\bm Z_{ij} S_{ij}^*$ property, which holds in \refcdfepf.  This property is used in the proof of Section \ref{sec:equiv}.
In the presence of the line loss equation \eqref{eq:cdf_e_3}, the following property holds,
\begin{subequations}
\begin{align}
& S_{ij}+S_{ji} = \bm Z_{ij} \left( L_{ij} - \bm {Y^c}_{ij} S_{ij} - \bm {Y^{c*}}_{ij} S^*_{ij} + |\bm {Y^c}_{ij}|^2 \frac{W_i}{|\bm T_{ij}|^2} \right) +  \bm {Y^{c*}}_{ij} \frac{W_i}{|\bm T_{ij}|^2} + \bm {Y^{c*}}_{ji} W_j  \\
& |\bm Z_{ij}|^2 \left( L_{ij} - \bm {Y^c}_{ij} S_{ij} - \bm {Y^{c*}}_{ij} S^*_{ij} + |\bm {Y^c}_{ij}|^2 \frac{W_i}{|\bm T_{ij}|^2} \right) = \bm Z_{ij}^* S_{ij} + \bm Z_{ij}^* S_{ji} - \bm Z_{ij}^* \bm {Y^{c*}}_{ij} \frac{W_i}{|\bm T_{ij}|^2} - \bm Z_{ij}^* \bm {Y^{c*}}_{ji} W_j 
\end{align}
%
When combined the voltage magnitude difference equation \eqref{eq:cdf_e_4}, yields, 
%
\begin{align}
& \left(1 + \bm Z_{ij}\bm {Y^c}_{ij} + \bm Z^*_{ij}\bm {Y^{c*}}_{ij} \right) \frac{W_i}{|\bm T_{ij}|^2} - W_j = \bm Z^*_{ij} S_{ij} + \bm Z_{ij} S^*_{ij} - |\bm Z_{ij}|^2 \left( L_{ij} - \bm {Y^c}_{ij} S_{ij} - \bm {Y^{c*}}_{ij} S^*_{ij} + |\bm {Y^c}_{ij}|^2 \frac{W_i}{|\bm T_{ij}|^2} \right) \\
& \left(1 + \bm Z_{ij}\bm {Y^c}_{ij} + \bm Z^*_{ij}\bm {Y^{c*}}_{ij} \right) \frac{W_i}{|\bm T_{ij}|^2}-W_j = \bm Z^*_{ij} S_{ij} + \bm Z_{ij} S_{ij}^* - \bm Z_{ij}^* S_{ij} - \bm Z_{ij}^* S_{ji}  + \bm Z_{ij}^* \bm {Y^{c*}}_{ij} \frac{W_i}{|\bm T_{ij}|^2} + \bm Z_{ij}^* \bm {Y^{c*}}_{ji} W_j  \\
& \left(1 + \bm Z_{ij}\bm {Y^c}_{ij} \right) \frac{W_i}{|\bm T_{ij}|^2} - \left(1 + \bm Z^*_{ij}\bm {Y^{c*}}_{ji} \right) W_j = \bm Z_{ij} S_{ij}^* - \bm Z_{ij}^* S_{ji} \\
&  \bm Z_{ij} S_{ij}^* = \bm Z_{ij}^* S_{ji} + \left(1 + \bm Z_{ij}\bm {Y^c}_{ij} \right) \frac{W_i}{|\bm T_{ij}|^2} - \left(1 + \bm Z^*_{ij}\bm {Y^{c*}}_{ji} \right) W_j 
\end{align}
\end{subequations}

\section{The Extended Convex DistFlow in Real Numbers}

In the interest of clean proofs, this document focuses on a complex representation of the extended CDF model.
In practice however, the a real number representation is often needed for implementation.
This section presents the real number version of \refcdfepf~to aid the implementations of this model.

First we develop a real number short-hand for $\bm Z^*_{ij} \bm{T}_{ij}$ as follows,
\begin{subequations}
\begin{align}
& \Re \left( \bm Z^*_{ij} \bm{T}_{ij} \right) = \bm {tz^R}_{ij}  = \bm r_{ij} \bm {t^R}_{ij} + \bm x_{ij} \bm {t^I}_{ij} \;\; (i,j) \in E \\
& \Im \left( \bm Z^*_{ij} \bm{T}_{ij} \right) = \bm {tz^I}_{ij}  = \bm r_{ij} \bm {t^I}_{ij} - \bm x_{ij} \bm {t^R}_{ij} \;\; (i,j) \in E
\end{align}
\end{subequations}
Notice that the superscripts $R,I$ are used for the real and imaginary parts of a complex number respectivly.
Expanding \eqref{eq:vv_factor_e} using this short-hand into real and imaginary components $\Re(V_iV^*_j), \Im(V_iV^*_j)$ we have,
\begin{subequations}
\begin{align}
& \Re(V_iV^*_j) = \left( \bm {t^R}_{ij} + \bm {tz^R}_{ij} \bm {g^c}_{ij} + \bm {tz^I}_{ij} \bm {b^c}_{ij} \right) \frac{v^2_i}{\bm t^2_{ij}} - \bm {tz^R}_{ij}p_{ij} + \bm {tz^I}_{ij}q_{ij}   \;\; (i,j) \in E \\
& \Im(V_iV^*_j) = \left( \bm {t^I}_{ij} + \bm {tz^I}_{ij} \bm {g^c}_{ij} - \bm {tz^R}_{ij} \bm {b^c}_{ij} \right) \frac{v^2_i}{\bm t^2_{ij}} - \bm {tz^I}_{ij}p_{ij} - \bm {tz^R}_{ij}q_{ij}  \;\; (i,j) \in E
\end{align}
\end{subequations}
both of which are used to implement the voltage angle difference constraints in any model with $v,p,q$ variables.  
The complete real number implementation of \refcdfepf~is presented in Model \ref{model:cdf_pf_e_real}.

\begin{model}[t]
\caption{The Extended CDF Relaxation of AC Power Flow \refcdfepf~in Real Numbers.}
\label{model:cdf_pf_e_real}
\begin{subequations}
\begin{align}
\mbox{\bf variables: \hspace{-1.3cm} } \nonumber \\
& p^g_i, q^g_i \;\; \forall i\in N \nonumber \\
& w_i \;\; \forall i\in N \nonumber \\
& l_{ij} \;\; \forall (i,j)\in E \nonumber \\
& p_{ij}, q_{ij} \;\; \forall (i,j)\in E \cup E^R \nonumber \\
\mbox{\bf subject to: \hspace{-1.3cm} } \nonumber \\
& (\bm {v^l}_i)^2 \leq w_i \leq (\bm {v^u}_i)^2 \;\; \forall i \in N \\
& \bm {p^{gl}}_i \leq p^g_i \leq \bm {p^{gu}}_i \;\; \forall i \in N \\
& \bm {q^{gl}}_i \leq q^g_i \leq \bm {q^{gu}}_i \;\; \forall i \in N \\
& p^g_i - {\bm p^d_i} - \bm g^s_{i} w_i = \sum_{\substack{(i,j)\in E \cup E^R}} p_{ij} \;\; \forall i\in N \\ 
& q^g_i - {\bm q^d_i} + \bm b^s_{i} w_i = \sum_{\substack{(i,j)\in E \cup E^R}} q_{ij} \;\; \forall i\in N \\ 
& \tan(-\bm {\theta^\Delta}) \left( \left( \bm {t^R}_{ij} + \bm {tz^R}_{ij} \bm {g^c}_{ij} + \bm {tz^I}_{ij} \bm {b^c}_{ij} \right) \frac{w_i}{\bm t^2_{ij}} - \bm {tz^R}_{ij}p_{ij} + \bm {tz^I}_{ij}q_{ij}  \right) \leq \nonumber \\
& \left( \bm {t^I}_{ij} + \bm {tz^I}_{ij} \bm {g^c}_{ij} - \bm {tz^R}_{ij} \bm {b^c}_{ij} \right) \frac{w_i}{\bm t^2_{ij}} - \bm {tz^I}_{ij}p_{ij} - \bm {tz^R}_{ij}q_{ij}  \;\; (i,j) \in E \\
& \tan(\bm {\theta^\Delta}) \left( \left( \bm {t^R}_{ij} + \bm {tz^R}_{ij} \bm {g^c}_{ij} + \bm {tz^I}_{ij} \bm {b^c}_{ij} \right) \frac{w_i}{\bm t^2_{ij}} - \bm {tz^R}_{ij}p_{ij} + \bm {tz^I}_{ij}q_{ij}  \right) \geq \nonumber \\ 
& \left( \bm {t^I}_{ij} + \bm {tz^I}_{ij} \bm {g^c}_{ij} - \bm {tz^R}_{ij} \bm {b^c}_{ij} \right) \frac{w_i}{\bm t^2_{ij}} - \bm {tz^I}_{ij}p_{ij} - \bm {tz^R}_{ij}q_{ij}  \;\; (i,j) \in E \\
& p_{ij} + p_{ji} = \bm r_{ij} \left( l_{ij} - 2( \bm {g^c}_{ij} p_{ij} - \bm {b^c}_{ij} q_{ij}) + (\bm {y^c}_{ij})^2 \frac{w_i}{\bm t_{ij}^2} \right) + \bm {g^c}_{ij} \frac{w_i}{\bm t_{ij}^2} + \bm {g^c}_{ji} w_j \;\; (i,j) \in E \\
& q_{ij} + q_{ji} = \bm x_{ij} \left( l_{ij} - 2( \bm {g^c}_{ij} p_{ij} - \bm {b^c}_{ij} q_{ij}) + (\bm {y^c}_{ij})^2 \frac{w_i}{\bm t_{ij}^2} \right) - \bm {b^c}_{ij} \frac{w_i}{\bm t_{ij}^2} - \bm {b^c}_{ji} w_j\;\; (i,j) \in E \\
& \left( 1 + 2(\bm r_{ij} \bm {g^c}_{ij} - \bm x_{ij} \bm {b^c}_{ij}) \right) \frac{w_i}{\bm {t}^2_{ij}}- w_j = \nonumber \\ 
& 2(\bm r_{ij}p_{ij} + \bm x_{ij}q_{ij}) - (\bm r^2_{ij} + \bm x^2_{ij}) \left( l_{ij} - 2( \bm {g^c}_{ij} p_{ij} - \bm {b^c}_{ij} q_{ij}) + (\bm {y^c}_{ij})^2 \frac{w_i}{\bm t_{ij}^2} \right) \;\; (i,j) \in E  \\
& p_{ij}^2 + q_{ij}^2 \leq \frac{w_i}{\bm {t}^2_{ij}}l_{ij} \;\; (i,j) \in E \\
& p_{ij}^2 + q_{ij}^2 \leq (\bm {s^u}_{ij})^2 \;\; (i,j) \in E \cup E^R
\end{align}
\end{subequations}
\end{model}

\end{document}